\documentclass[11pt]{amsart}
\usepackage{amssymb,amsmath}
\usepackage{microtype}
\usepackage{bbm}
\usepackage{bm}
\usepackage[]{graphicx}
\usepackage{color}
\usepackage{amssymb, epsfig}
\usepackage[toc,page]{appendix}

\numberwithin{equation}{section}

\numberwithin{figure}{section}
\newtheorem{theorem}{Theorem}[section]
\newtheorem{lemma}[theorem]{Lemma}

\newtheorem{remark}[theorem]{Remark}
\begin{document}
\title[ Malliavin calculus for the stochastic heat equation and results on the density ] {Malliavin calculus for the stochastic heat equation and results on the density}

%%%%%%%%%%%%%%%%%%%%%%%%%%%%%%%%%%%%%%%%%%%%%%%%%%

%\author[Dimitris. Farazakis]{Dimitris Farazakis$^{*}$}

%\author[Georgia. Karali]{Georgia Karali$^{\dag *}$}
%\author[Alexandra. Stavrianidi]{Alexandra Stavrianidi$^{\# }$}

\author[Farazakis, Karali, Stavrianidi]{Dimitris Farazakis$^{*}$, Georgia Karali$^{\dag *}$, 
Alexandra Stavrianidi $^{\#}$ }

\thanks
{$^{\dag}$ Department of Mathematics, National and Kapodistrian University of Athens, Panepistimiopolis, Athens, Greece.}

\thanks
{$^{\#}$ Department of Mathematics, Stanford University, USA}
\thanks
{$^{*}$ Institute of Applied and Computational Mathematics,
FORTH, GR--711 10 Heraklion, Greece.}

%
%
%

%\begin{AMS}
%\end{AMS}
%\date{Received: date / Revised: date}
%\date{today}
%
%

\begin{abstract}
We study the one-dimensional stochastic heat equation with unbounded, nonlinear, Lipschitz coefficients with Dirichlet boundary conditions. Using Malliavin calculus, we construct a piecewise approximation of the solution $u$ and establish regularity results. This approximation enables us to provide a new proof of the existence of a density for the random variable $u(t,x)$ at any fixed $t,x$. Unlike existing proofs, which rely on comparison principles (\cite{muellernualart}, \cite{fjrfhrug}), our approach is based purely on a localization argument, which allows us to handle the unbounded coefficients.
\end{abstract}
\maketitle \pagestyle{myheadings}
\thispagestyle{plain}

\thispagestyle{plain}
%`
%
%
%%%%%%%%%%%%%%%%%%%%%%%%%%%%%%%%%%%%%%%%%%%%%%%%%%%%%%%%%%%%%%%%%%%%

\section{Introduction} \label{sbi}
\subsection{The stochastic problem}

We apply Malliavin calculus to the stochastic heat equation in one dimension
\begin{equation} \label{sm0}
\frac{\partial u}{\partial t}=\Delta u+f(t,x,u(t,x))+\sigma(t,x,u(t,x))W(dt,dx),
\end{equation}
where $x \in [0,1],t \in [0,T]$, and we prescribe the Dirichlet boundary conditions, $u(t,0)=u(t,1)=0$,
and an initial condition $u_0 \in C([0,1])$, with $u_0(0)=u_0(1)$. 
The functions $f, \sigma : [0,T] \times [0,1]\times \mathbb{R} \to \mathbb{R}$ are infinitely differentiable with bounded derivatives,  and the diffusion coefficient $\sigma$ is lower bounded by a positive constant. We denote by $W(dt, dx)$ in the forcing term the space-time white noise in the sense of Walsh, \cite{W6}, given as the formal derivative of a Wiener process. 
This stochastic partial differential equation arises in various fields such as population dynamics, genetics, statistical physics, and others. A key tool in our analysis is Malliavin calculus, an extension of the calculus of variations in the stochastic setting that allows us to study the smoothness of the law of the random variable $u(t,x)$ by making precise the notion of differentiation with respect to the white noise. Using tools from Malliavin calculus we can prove that the law of the solution $u(t,x)$ at a specific point $(t,x)$ has a density. We refer the reader to the books \cite{N}, \cite{dalang} for a comprehensive presentation of Malliavin calculus and of stochastic partial differential equations, respectively.

% $W$ will be expressed as follows
% $$dW:=W(dt,dx).$$ 

%$W$ is a space-time white noise on the probability space $(\Omega,\mathcal{F},P)$ equipped with filtration generated by $W$ %such that %
%$$\mathcal{F}^{W}_t:=\sigma(\{W(A) \mbox{ }| \mbox{ } A \in \mathcal{B}([0,t] \times[0,1])\}) \vee \mathcal{N}$$
%for the $\mathbb{P}$-null sets, $\mathcal{N}$. This is a Walsh noise, $\dot{W}$. That is, we define a two-parameter (one %dimensional) Wiener process, $W:=\{W(t,x), \mbox{ } t \in [0,T], \mbox{ } x \in [0,1]\}$, and then the differential form of

The stochastic heat equation has been studied in \cite{fjrfhrug} in the sense of Malliavin calculus, where the authors impose growth conditions on the coefficients $f,\sigma$ and provide a necessary and sufficient condition for the existence of a density for $u(t,x)$ for all $t>0$. They show that the law of the random variable $u(t,x)$ is absolutely continuous with respect to the Lebesgue measure under a nondegeneracy condition on $\sigma$. Their argument relies on a Galerkin approximation by solutions to stochastic ordinary differential equations and on a comparison principle.
In \cite{pardonuaa}, the authors study the stochastic heat equation with autonomous coefficients $f, \sigma$ and show that the law of any vector $\left(u\left(t, x_{1}\right), \ldots, u\left(t, x_{d}\right)\right), d \in \mathbb{N}, 0 \leqslant x_{1} \leqslant \cdots \leqslant x_{d} \leqslant 1, t>0$, has a smooth, strictly positive density with respect to Lebesgue measure, under a stronger nondegeneracy assumption on $\sigma$. In \cite{muellernualart}, the authors prove the smoothness of the density for the solution to $(\ref{sm0})$ with $f,\sigma$ that have linear growth and are Lipschitz. This proof relies on a comparison principle and on the existence of negative moments of all orders for the linear stochastic heat equation with multiplicative noise. 
%In \cite{pardonuaa} they show the existence HOW??? {\color{blue}{explain how they get the density}}.

\smallskip

\noindent The aim of this paper is to provide a different proof for the existence of a density for the random variable $u(t,x)$ for any $t\in (0,T), x \in (0,1)$, where $u$ is the solution to the nonlinear, multiplicative stochastic heat equation with unbounded, Lipschitz coefficients $(\ref{sm0})$. Our novel proof introduces a localization argument, which allows us to work directly with the nonlinear stochastic heat equation with unbounded, Lipschitz coefficients without employing comparison principles or imposing additional growth conditions. We also establish regularity results for the Malliavin derivative of the solution $u$ in the local sense.

%MOTIVATION AND MAIN RESULTS
\subsection{ Definition of the mild solution and setup}

We work on a complete probability space $(\Omega, \mathcal{F}, \mathbb{P})$, with $W$ a space-time white noise. This space is equipped with the filtration generated by $W,  \mathcal{F}_{t}^{W}$, which can be written as

$$
\mathcal{F}_{t}^{W}:=\sigma(\{W(A) \mid A \in \mathcal{B}([0, t] \times[0,1]\}) \cup \mathcal{N},
$$
where $\mathcal{N}$ here denotes the $\mathbb{P}$-null sets. \\

\noindent The solution to $(\ref{sm0})$ in the weak sense must satisfy for every $t \ge0$ and every $\phi \in C^{1,2}([0,T]\times[0,1])$  with $\phi(s,0)=\phi(s,1)=0,$ for every $s \in [0,t]$,
\begin{equation} \label{weak123}
\begin{split}
\int_{0}^{1}u(t,x)\phi(t,x) dx &=\int_{0}^{1} u(0,x)\phi(0,x)dx+\int_{0}^{t}\int_{0}^{1} u(s,x)\frac{\partial^2\phi}{\partial x^2}(s,x)dx ds \\
&+\int_{0}^{t}\int_{0}^{1} u(s,x)\frac{\partial\phi}{\partial t}(s,x)dx ds +\int_{0}^{t}\int_{0}^{1}f(s,x,u(s,x))\phi(s,x)dxds\\
&+\int_{0}^{t}\int_{0}^{1} \sigma(s,x,u(s,x))\phi(s,x) W(ds,dx).
\end{split}
\end{equation}

\noindent We now provide the definition of a mild solution $u$ of $(\ref{sm0})$.
The mild solution of $(\ref{sm0})$ satisfies, for any $x \in [0,1]$ and $t \in [0,T]$,
\begin{equation} \label{smr00}
\begin{split}
u(t,x) &=\int_{0}^{1} G_t(x,y)u_0(y)dy+\int_{0}^{t}\int_{0}^{1} G_{t-s}(x,y) f(s,y,u(s,y)) dy ds \\
&+ \int_{0}^{t}\int_{0}^{1}G_{t-s}(x,y) \sigma(s,y,u(s,y))W(dy,ds),
\end{split}
\end{equation}
where the heat kernel with Dirichlet boundary conditions on $[0,1]$ is given by 
$$
 G_{t}(x,y)=\sum_{n=1}^{\infty} 2e^{-\pi^{2} n^{2} t} \sin(n\pi x) \sin(n \pi y),
 $$
 where $ t\in [0,T], x, y \in[0, 1].$

\subsection{Main results} 
In this work, we present a new proof for the existence of a density for the random variable $u(t,x)$ for any $t\in (0,T), x\in (0,1)$, where $u$ is the solution to the stochastic heat equation with unbounded, Lipschitz coefficients. Our localization approach is inspired by techniques applied in \cite{weber} to the stochastic Cahn-Hilliard equation with bounded noise diffusion and in \cite{AFK} to the stochastic Cahn-Hilliard/Allen-Cahn equation with unbounded noise diffusion in one dimension. Unlike these models, our equation involves the heat kernel, whose singularity near the origin is of higher order. This creates significant analytical challenges, as the standard approach of these previous works involving the $L^{\infty}$ norm fails, requiring us to work with the $L^{2}$ norm instead, and the analysis becomes far more involved technically. In order to establish the positivity of the norm for the Malliavin derivative of $u$ in the proof of the existence of density we need to take advantage of cancellations that arise after obtaining tight bounds.

%More precisely, we construct a localization for $u$ by a sequence $u_{n}$, for which we prove existence of the Malliavin derivative $D_{s,y}u_n$ and the strict positivity condition for a norm of $D_{s,y}u_n$, which can extend to $D_{s,y}u$ and guarantee the existence of a density. 
%

We begin by applying Malliavin calculus on the stochastic heat equation to prove that $ u \in L^{loc}_{1,2}$, i.e., that $u$ has a local version in $L_{1,2}$ in the sense of ~\cite{N}, p. 49 (there exists a sequence 
 $\{ (\Omega_n, u_n ); \Omega_{n} \subseteq \Omega \}_{n \ge 1} \subset \mathcal{F} \times L_{1,2}$  such that $\Omega_n \uparrow \Omega$ a.s.  and 
 $u=u_n$ a.s. on $\Omega_n $). %and then the Malliavin derivative $D_{s,y}u$ is defined without ambiguity as  $D_{s,y}u= D_{s,y} u_n$ on $\Omega_n$ %.
 This implies that $u$ has a local version in $D_{1,2}$ as well. We can then establish the existence of a density for the random variable $u(t,x)$ with $t \in (0,T), x \in (0,1)$ by proving the strict positivity condition of Theorem 2.1.3 from \cite{N}. This localization is not present in (\cite{fjrfhrug}, \cite{muellernualart}, \cite{pardonuaa}), and the proof for the existence of a density in these works is different from ours. In particular, the proofs in (\cite{fjrfhrug}, \cite{muellernualart}) rely on comparison principles to establish the positivity of the norm. In contrast, we work directly with the nonlinear stochastic heat equation and employ a localization method that allows us to handle unbounded, Lipschitz coefficients without imposing additional growth or boundedness conditions on $f,\sigma$.

\smallskip

In particular, we prove the following Main Theorem.
\begin{theorem} \label{theorem21}
{Let $u$ be the mild solution to the stochastic heat equation $(\ref{smr00})$ in one dimension, with zero Dirichlet boundary conditions and initial condition  $u_0$.

Let $f$ and $\sigma$ satisfy:

(1) $f,\sigma$ are infinitely differentiable with bounded derivatives. In particular, this means that $f,\sigma$ are globally Lipschitz in the $u$ variable. Thus, for all $ t\in [0,T], x \in [0,1], u,v \geq 0$ there exist constants $K_{f},K_{\sigma}$ such that 

\begin{equation} \nonumber
\begin{split} 
&|f(t,x,u)-f(t,x,v)| \leq K_{f}|u-v|,  \\
&|\sigma(t,x,u)-\sigma(t,x,v)| \leq K_{\sigma}|u-v|. \\
\end{split}
\end{equation}

(2) $\sigma$ satisfies a nondegeneracy condition. There exists a constant $c_0>0$ so that

$$|\sigma(t,x,y)| \geq c_0 \text{ for all } t \in [0,T], x \in [0,1], y \in \mathbb{R}.$$

(3) $u_{0} \in C([0,1])$, with $u_{0}(0)=u_{0}(1)=0$. \\

\noindent Then, $u$ has a local version in the sense of ~\cite{N}, and  belongs to the space $L_{1,2}^{loc} \subset D_{1,2}^{loc}$. Moreover, the law of the random variable $u(t,x)$ for $t \in (0,T), x \in (0,1)$ is absolutely continuous with respect to the Lebesgue measure on $\mathbb{R}$.
}
\end{theorem}

\noindent The paper is organized as follows. In section 2, we cover some preliminaries. In section 3, we introduce the localization $u_n$, for which we show existence and uniqueness. In section 4, we apply Malliavin Calculus to prove that $u \in L_{1,2}^{loc}$. In section 5, we prove a positivity condition for the norm of $D_{s,y}u_n$, from which we can establish the positivity condition for the norm of $D_{s,y}u$ and infer the existence of a density for the random variable $u(t,x)$ for any $t \in (0,T), x\in (0,1)$. Throughout the paper, constants $C$ and $c$ may change from one line to the next and we denote by $C_{T}$ any constant that depends on the final time $T$.

 %EXISTENCE ND UNIQUENSS ORIGINAL HEAT

%%%%%%%%%%%%%%%%%%%%%%%
%THE HEAT KERNEL

\section{Preliminaries} 

\bigskip
\subsection{A heat kernel estimate}
We will need the following lemma for the heat kernel, $G_{t-s}(x,y),$ for $t \in [0,T]$ and $ x, y \in [0,1]$.
\begin{lemma} \label{O}
Consider a function $H(s,t;x,y)$ with $0 \le s<t \le T$, $x,y \in \mathbb{R}$, for which there exist positive constants $K,a, b, c$  such that
$$H(s,t;x,y) \le K\frac{1}{|t-s|}exp\Big ( -\alpha \frac{|x-y|^2}{|t-s|}\Big )$$
for all $0 \le s<t \le T$, $x,y \in \mathbb{R}$. Let the linear operator
$$J(\upsilon)(t,x):=\int_{0}^{t}\int_{0}^{1}H(s,t;x,y)\upsilon(s,y)dyds$$ for $t \in [0,T], \mbox{ } x \in [0,1]$, and every $\upsilon \in L^{\infty}([0,T];L^1([0,1])).$ Then the following results hold for  $\rho \in [1,\infty], q \in [1,\rho)$ and $k:=1+\frac{1}{\rho}-\frac{1}{q}.$ \\
\begin{enumerate}
\item $J$ is a bounded linear operator from $L^{\gamma}([0,T];L^q([0,1]))$ into $C([0,T];L^p([0,1])$ for $\gamma>2k^{-1}$.\\
\item For all $t \in [0,T]$, $k>0$ and $\gamma>2k^{-1}$ 
$$\|J(\upsilon)(t,\cdot)\|_{\rho} \le \int_{0}^{t}(t-s)^{\frac{k}{2}-1}\|\upsilon(s,\cdot)\|_{q}ds.$$
\end{enumerate}
In addition, taking $H(s,t;x,y)=G_{t-s}^2(x,y)$, to be the Dirichlet heat kernel, we have that  for all $s<t \in [0,T]$ and $p,q>1$ with $\frac{1}{p}+\frac{1}{q}=1$
$$J(\upsilon)(t) \le \int_{0}^{t}\frac{1}{ (t-s)^{1-\frac{1}{2p}}}\|\upsilon(s,\cdot)\|_{q}ds.$$

\end{lemma}
\proof
Conditions (1) and (2) have been proven in the corollary (3.4) of \cite{Gyongy}. Let us prove the last part, that is when $H(s,t;x,y)=G_{t-s}^2(x,y)$. Firstly, we can bound the function $G_{t}(x,y)$ by the heat kernel on $\mathbb{R}$.
We can write:
$$
G_{t}(x,y) \leq \frac{ e^{-\frac{(x-y)^2}{4t}}}{\sqrt{4 \pi t}}.
$$
 So we have the estimate for $p,q>1$ satisfying that $\frac{1}{p}+\frac{1}{q}=1$:
 \begin{equation} \nonumber
 \begin{split}
J(\upsilon)(t) &\le \int_{0}^{t} \int_{0}^{1} \frac{e^{-\frac{(x-y)^2}{2(t-s)}}}{4\pi (t-s)} \upsilon(s,y) d y d s \leq \int_{0}^{t} \frac{1}{4 \pi (t-s)} \int_{0}^{1} e^{-\frac{(x-y)^2}{2 (t-s)}} \upsilon(s,y) d y d s \\
& \leq \int_{0}^{t} \frac{1}{4 \pi (t-s)} \left(\int_{0}^{1} (e^{-\frac{(x-y)^2}{2 (t-s)}})^{p} d y\right)^{\frac{1}{p}}  \left(\int_{0}^{1} (\upsilon(s,y) )^{q} d y\right)^{\frac{1}{q}} d s.
\end{split}
\end{equation}

Moreover, we have that
 
\begin{equation} \nonumber
\int_{0}^{1} (e^{-\frac{(x-y)^2}{2 (\tau-s)}})^{p} d y 
\leq \sqrt{2(\tau-s)} \frac{1}{\sqrt{p}}\int_{-\infty}^{+\infty} e^{-z^2} d z  \leq C \sqrt{(\tau-s)} 
\end{equation}
for a constant $C=\sqrt{\frac{2\pi}{p}}$.
So, in the end, we get 
$$
\begin{aligned}
&J(\upsilon)(t)  \leq 
C \int_{0}^{t} \frac{1}{4 \pi (t-s)} 
(\sqrt{t-s})^{\frac{1}{p}}\left(\int_{0}^{1} (\upsilon(s,y))^{q} d y\right)^{\frac{1}{q}} d s \\
& =C \int_{0}^{t} \frac{1}{ (t-s)^{1-\frac{1}{2p}}} 
\left(\int_{0}^{1} (\upsilon(s,y))^{q} d y\right)^{\frac{1}{q}} d s=  C \int_{0}^{t} \frac{1}{ (t-s)^{1-\frac{1}{2p}}} 
\| \upsilon (s,\cdot) \|_{L^q} d s.
\end{aligned}
$$

%\begin{remark} \label{infinitycase}
%We denote the sup-norm  by $\|\cdot\|_{\infty}$, and  for $f : [0,T]\times[0,1] \to \mathbb{R}$, 

%\begin{equation} \nonumber 
%\begin{split}
%&\|f(t,\cdot)\|_{\infty}:=\sup_{x \in [0,1]}|f(t,x)| \\
%&\|f(\cdot,x)\|_{\infty}:=\sup_{t \in [0,T]}|f(t,x)|.
%\end{split}
%\end{equation}
%\end{remark}
%%%%%%%%%%%%%%%%%%%%%%%%%%%

%THE LOCALIZATION SPACE

\subsection{Regularity \& localization of a stochastic process $u$}
In this section, we construct the localization space $\Omega_n$ of $\Omega$ by choosing an appropriate cut-off function for $f$ and $\sigma$.  For the continuous stochastic process $u$, we can localize the functions $f:=f(t,x,u(t,x)), \sigma:=\sigma(t,x,u(t,x))$ where $f : [0,T] \times [0,1] \times \mathbb{R} \to \mathbb{R}$ and $\sigma : [0,T] \times [0,1] \times \mathbb{R} \to \mathbb{R}$ as in (\cite{AFK}, \cite{AFM}, \cite{weber}). That is, suppose that $H_n: [0, +\infty] \rightarrow [0, +\infty]$ is a  $C^1$ cut-off sequence satisfying for any $n>0$ that $|H_n|\le 1 \mbox{ and }|H^\prime_n|\le 2,$ with
\begin{equation} \label{truncatedfction451}
H_n(x):=
\begin{cases}
1  &     \textrm{for  $|x|< n $ }   \\
0 & \textrm{for $|x|\ge n+1 $}.
\end{cases}
\end{equation}
 Then $H_n(|y|)f(t,x,y), H_n(|y|)\sigma(t,x,y)$ are $C^1$ functions with bounded derivatives. The derivatives of $H_{n}(|y|)f(t,x,y)$ (and of $H_{n}(|y|)\sigma(t,x,y)$ respectively) are bounded by a constant that depends on $n$ and on the Lipshitz constants for $f,\sigma$.
 
 \begin{remark}
 It is well known (for instance from Theorem 2.4.3 in \cite{N}) that $(\ref{smr00})$ has a unique adapted solution $\{u(t, x) ; t\in [0,T], x\in [0,1]\}$ with continuous paths.
 We want to localize the continuous random variable $u:=u(t,x)$ for $t \in [0,T]$, $x \in [0,1]$ using a truncation argument, as in \cite{AFK}, \cite{AFM}, \cite{weber}.
\end{remark}
We define the spaces
\begin{equation} \label{Omega4}
\Omega_{n}:= \Big{\{} \omega\in\Omega:\;\sup_{t \in [0,T]}\sup_{x \in [0,1]} |u(t,x;\omega)| < n
\Big{\}}
\end{equation}
such that $\Omega_1\subseteq\Omega_2\subseteq\cdots\subseteq \Omega.$
We apply a piecewise approximation on $(\ref{smr00})$ by considering
\begin{equation} \label{mild0}
\begin{split}
u_n(t,x) &=\int_{0}^{1} G_t(x,y)u_0(y)dy+\int_{0}^{t}\int_{0}^{1} G_{t-s}(x,y) H_n(|u_n(s,y)|)f(s,y,u_n(s,y))) dy ds \\
&+ \int_{0}^{t}\int_{0}^{1}G_{t-s}(x,y)H_n(|u_n(s,y)|)\sigma(s,y,u_n(s,y))W(dy,ds).
\end{split}
\end{equation}
Now we define the localization spaces for the random process as in (\cite{N}, \cite{AFK}). 

\noindent The space $D_{1,2}$: A process $u(t,x) \in L^{2}(\Omega \times [0,T] \times [0,1])$ belongs to the space $D_{1,2}$ if its Malliavin derivative exists and $u$ satisfies for any $t \in [0,T], x \in [0,1]$:   

\begin{equation} \label{fokfogiiggot9g}
\|u\|^2_{D_{1,2}}=\mathbf{E} (|u|^2)+\mathbf{E} \Big (\|D_{\cdot,\cdot}u\|^2_{L^2 \left (  [0,T] \times [0,1] \right )} \Big )< \infty,
\end{equation}
where
\begin{equation} \label{norm3}
\|D_{\cdot,\cdot}u(t,x)\|^2_{L^2 \left ([0,T] \times [0,1] \right )}=\int_{0}^{T} \int_{0}^{1} |D_{s,y} u(t,x)|^2 dyds.
\end{equation}

Note that $D_{1,2}$ is a Hilbert space with inner product
$$<f,g>:=\mathbf{E}(fg)+\mathbf{E}(<D_{\cdot,\cdot}f,D_{\cdot,\cdot}g>_{L^2 \left (  [0,T] \times [0,1]\right )}).$$

%Thus, $D_{1,2}$ is the
%set of random variables $u$ such that the Malliavin derivative
 %$D_{s,y}u(t,x)$ exists, for any
%$y\in [0,1]$ and any $s\geq 0$ and any
%$(t,x) \in ([0,T] \times [0,1])$, and satisfies the two previous relations. \\

%%%%%%L_1,2

\noindent The space $L_{1,2}$: A process $u(t,x) \in L^2(\Omega \times   [0,T] \times [0,1])   $ belongs to the space  $L_{1,2}$ \index{$L_{1,2}$} if  $u\in D_{1,2}$ and  the Malliavin derivative $D_{s,y}u(t,x)$ satisfies
\begin{equation} \label{SJ29d}
\mathbf{E} \bigg (\int_{0}^{T} \int_{0}^{1} \int_{0}^{T} \int_{0}^{1} |D_{s,y} u(t,x)|^2 dyds dxdt \bigg )< \infty.
\end{equation}

We note that in the above definitions $u \in L^2(\Omega \times  [0,T] \times [0,1]) $ means that
$$\|u\|_{L^2(\Omega \times  [0,T] \times [0,1]) }:=
\Big{(}\mathbf{E}\Big{(}\int_0^T\int_{0}^{1}|u(t,x)|^2dxdt\Big{)}\Big{)}^{1/2}<\infty.$$
%\begin{remark} \label{newcareersl0}
%If $u$ belongs to the space $L_{1,2}$, then $u$ belongs to the space $D_{1,2}$ as well. In addition, the localization of the domains $D_{1,2}$ and $L_{1,2}$ (i.e.,., the spaces $D^{loc}_{1,2}$ and $L^{loc}_{1,2}$) demand the existence of appropriate stochastic process $u_n$ which firstly it belongs to the general spaces $L_{1,2}$ and $D_{1,2}$ respectively, and then it satisfies the specified conditions each of local definition (see below). 
%\end{remark}

%%%%%%%D_loc 1,2
\noindent The space $D^{loc}_{1,2}$: A random variable  $\{u(t,x)\}_{(t,x) \in ( [0,T] \times [0,1])}$ belongs to the space $D^{loc}_{1,2}$ if there exists a pair of sequences  $\{ (\Omega_n, u_n ); \Omega_{n} \subseteq \Omega \}_{n \ge 1} \subset \mathcal{F} \times D_{1,2}$, that satisfies
\begin{enumerate}
\item $\Omega_n \uparrow \Omega$ a.s, 
\item $u=u_n$ a.s on $\Omega_n. $
\end{enumerate}
%%%%%%%L_loc
\noindent The space $L^{loc}_{1,2}$: A random variable $\{u(t,x)\}_{(t,x) \in ( [0,T] \times [0,1])}$ belongs to the space $L^{loc}_{1,2}$ if there exists a pair of sequences  $\{ (\Omega_n, u_n ); \Omega_{n} \subseteq \Omega \}_{n \ge 1} \subset \mathcal{F} \times L_{1,2}$, that satisfies
\begin{enumerate}
\item $\Omega_n \uparrow \Omega$ a.s, 
\item $u=u_n$ a.s. on $\Omega_n. $
\end{enumerate}

\begin{remark}\label{malliavin}
According to these definitions of the spaces $L_{1,2}$, $D_{1,2}$, and their localizations, it follows that $L^{loc}_{1,2} \subseteq D^{loc}_{1,2}$. In addition, $L_{1,2} \subset D_{1,2}$, and then a localization of the form $(\Omega_n, u_n)$ of $u$ in $L_{1,2}$ is also a localization into the space $D_{1,2}$. The Malliavin derivative of $u$ is well defined through the Malliavin derivative of $u_n$, i.e., $D_{s,y}u_n=D_{s,y}u$ a.s. on $\Omega_n$.

\end{remark}

%%%%%%%%%%%%%%%%%%%%%%%
\section{Existence \& uniqueness} \label{similarres}

\begin{lemma} \label{Ox5} Let $f, \sigma$ be continuously differentiable functions with bounded derivatives. Then, the equation $(\ref{mild0})$ has a unique solution $u_{n}$. Additionally the following condition holds, 
\begin{equation} \label{nondegenerate0}
\sup_{t \in [0,T]}\sup_{x \in [0,1]} \mathbf{E} (|u_n(t,x)|^{2})<\infty.
\end{equation}
\end{lemma}
\proof 
We will construct a Cauchy sequence through a Picard iteration scheme which will converge to the solution $u_{n}$ of $(\ref{mild0})$. For given $n$ we define  
\begin{equation} \nonumber
u_{n,0}(t,x):=G_tu_0(x)
\end{equation} 
and then iteratively $(\ref{mild0})$, for $(t,x) \in (  [0,T]\times [0,1] )$, $k \ge 1$ and $s \ge 0$ such that $s < t$, 
\begin{equation} \label{statespotkk}
\begin{split}
u_{n,k+1}(t,x) &=\int_{0}^{1} G_t(x,y)u_0(y)dy+\int_{0}^{t}\int_{0}^{1} G_{t-s}(x,y) H_n(|u_{n,k}(s,y)|)f(s,y,u_{n,k}(s,y)) dy ds \\
&+ \int_{0}^{t}\int_{0}^{1}G_{t-s}(x,y)H_n(|u_{n,k}(s,y)|)\sigma(s,y,u_{n,k}(s,y))W(dy,ds).
\end{split}
\end{equation}
Then for $k \ge 1$,
\begin{equation} \nonumber
u_{n,k+1}(t,x)-u_{n,k}(t,x) = \int_{0}^{t}\int_{0}^{1}G_{t-s}(x,y) 
  (\hat{H}_f )dyds + \int_{0}^{t}\int_{0}^{1}G_{t-s}(x,y) (\hat{H}_{\sigma} )W(dy,ds)
\end{equation} 
for
$$\hat{H}_f:=H_n(|u_{n,k}(s,y)|)f(s,y,u_{n,k}(s,y))-H_n(|u_{n,k-1}(s,y)|)f(s,y,u_{n,k-1}(s,y)),$$
$$\hat{H}_{\sigma}:=H_n(|u_{n,k}(s,y)|)\sigma(s,y,u_{n,k}(s,y))- 
H_n(|u_{n,k-1}(s,y)|)\sigma(s,y,u_{n,k-1}(s,y)).$$

We get 
\begin{equation} \nonumber
|u_{n,k+1}(t,x)-u_{n,k}(t,x)|\le \int_{0}^{t}\int_{0}^{1}|G_{t-s}(x,y)| 
  |\hat{H}_f  |dyds +\Big | \int_{0}^{t}\int_{0}^{1}G_{t-s}(x,y)  (\hat{H}_{\sigma} )W(dy,ds)\Big |.
\end{equation}
Then for some $c>0$,
\small
\begin{equation} \nonumber
\mathbf{E} \left(|u_{n,k+1}(t,x)-u_{n,k}(t,x)|^{2} \right) \le c \mathbf{E} \left(\Big ( \int_{0}^{t}\int_{0}^{1}|G_{t-s}(x,y)||\hat{H}_f| dyds \Big )^{2} \right) +c \mathbf{E} \left( \Big | \int_{0}^{t}\int_{0}^{1}G_{t-s}(x,y) (\hat{H}_{\sigma})W(dy,ds)\Big |^{2} \right).
\end{equation}
Applying the Burkholder-Davis-Gundy inequality on the stochastic term then gives us that
\begin{equation} \nonumber
\mathbf{E}  (|u_{n,k+1}(t,x)-u_{n,k}(t,x)|^{2}  )\leq c \mathbf{E} \Big ( \Big ( \int_{0}^{t}\int_{0}^{1}|G_{t-s}(x,y)||\hat{H}_{f}| dyds \Big )^{2} \Big )+c \mathbf{E} \Big (\int_{0}^{t}\int_{0}^{1}(G_{\tau-s}(x,y))^2 (\hat{H}_{\sigma})^{2} dyds \Big )
\end{equation}
and 
\begin{equation} \label{p_{n,k}}
\begin{split}   
& \sup_{x \in [0,1]}\mathbf{E} (|u_{n,k+1}(t,x)-u_{n,k}(t,x)|^{2}) \leq c \sup_{x \in [0,1]} \mathbf{E} \Big ( \Big ( \int_{0}^{t}\int_{0}^{1}|G_{t-s}(x,y)||\hat{H}_{f}| dyds \Big )^{2} \Big )\\
&+c \sup_{x \in [0,1]}\mathbf{E} \Big ( \Big(\int_{0}^{t}\int_{0}^{1}(G_{\tau-s}(x,y))^2 (\hat{H}_{\sigma})^{2} dyds\Big )\Big ) := c\sup_{x \in [0,1]}\Lambda_{1}(t,x)+ c\sup_{x \in [0,1]}\Lambda_{2}(t,x).
\end{split}
\end{equation}
%THE TERM Q1 Q1 Q1 Q1 Q1 Q1 Q1 Q1 Q1 Q1 Q1 Q 1 Q1Q 1
We set $P_{n,k}(t)=\sup_{ y\in [0,1]} \mathbf{E} \bigg(|u_{n,k}(t,y)-u_{n,k-1}(t,y)|^{2}\bigg)$.  \\
We first bound the term $\Lambda_1(t,x)$. As we mentioned before, the truncated function $H_{n}(|y|)f(t,x,y)$ is Lipschitz in the $y$-variable. We have
\begin{equation} \nonumber
\begin{split}
|\hat{H}_f|&:=|H_n(|u_{n,k}(s,y)|)f(s,y,u_{n,k}(s,y))-H_n(|u_{n,k-1}(s,y)|)f(s,y,u_{n,k-1}(s,y))| \\
&\le |u_{n,k}(s,y)-u_{n,k-1}(s,y)|.
\end{split}
\end{equation}
Hence,
\begin{equation} \nonumber
\Lambda_1(t,x) \leq  c \mathbf{E} \Big ( \Big ( \int_{0}^{t}\int_{0}^{1}|G_{t-s}(x,y)| 
  |u_{n,k}(s,y)-u_{n,k-1}(s,y)|dyds \Big )^{2} \Big ).
\end{equation}

We have that 
\begin{equation} \nonumber
\begin{split}
&\mathbf{E} \Big ( \Big ( \int_{0}^{t}\int_{0}^{1}|G_{t-s}(x,y)| 
  |u_{n,k}(s,y)-u_{n,k-1}(s,y)|dyds \Big )^{2} \Big )\leq \\
&\mathbf{E}  \Big ( \int_{0}^{t}  \frac{1}{t-s}\int_{0}^{1}e^{-\frac{(x-y)^{2}}{2\pi (t-s)}} dy ds \int_{0}^{t} \int_{0}^{1}
  |u_{n,k}(s,y)-u_{n,k-1}(s,y)|^{2} dyds \Big ) \leq  \\
&\mathbf{E} \Big ( \int_{0}^{t}  \frac{1}{\sqrt{t-s}} ds \int_{0}^{t} \int_{0}^{1}
  |u_{n,k}(s,y)-u_{n,k-1}(s,y)|^{2} dyds \Big ) \\
  &\leq C_{T} \int_{0}^{t} \int_{0}^{1}  \mathbf{E} \Big ( 
  |u_{n,k}(s,y)-u_{n,k-1}(s,y)|^{2} \Big) dy  ds \\
  &\leq C_{T} \int_{0}^{t} \sup_{y \in [0,1]} \mathbf{E} \Big ( 
  |u_{n,k}(s,y)-u_{n,k-1}(s,y)|^{2} \Big)  ds \\
  & \leq C_{T} \int_{0}^{t} P_{n,k}(s) ds\\
  & \leq C_{T} \left( \int_{0}^{t} P_{n,k}(s)^{q} ds \right)^{\frac{1}{q}},\\
\end{split}
\end{equation}

\noindent where we take an exponent $q$ such that $\frac{1}{p}+\frac{1}{q}=1$ for some $p \in (1,2)$.

\noindent Now we bound the term $\Lambda_2(t,x)$. We have that the truncated function $H_{n}(|y|)\sigma(t,x,y)$ is Lipschitz in the $y$-variable, thus

\begin{equation} \nonumber
\begin{split}
|\hat{H}_{\sigma}|&:=|H_n(|u_{n,k}(s,y)|)\sigma(s,y,u_{n,k}(s,y))-H_n(|u_{n,k-1}(s,y)|)\sigma(s,y,u_{n,k-1}(s,y))| \\
&\le |u_{n,k}(s,y)-u_{n,k-1}(s,y)|.
\end{split}
\end{equation}

Hence,
\begin{equation} \nonumber
\begin{split}
\Lambda_2(t,x) \le c \mathbf{E} \Big (\int_{0}^{t}\int_{0}^{1}|G_{t-s}(x,y)|^2  |u_{n,k}(s,y)-u_{n,k-1}(s,y) |^{2} dyds\Big ).
\end{split}
\end{equation}

We have that

\begin{equation} \nonumber
\begin{split}
&\mathbf{E} \Big (\int_{0}^{t}\int_{0}^{1}|G_{t-s}(x,y)|^2  |u_{n,k}(s,y)-u_{n,k-1}(s,y) |^{2} dyds\Big )\leq 
c\mathbf{E} \Big (\int_{0}^{t}\int_{0}^{1}\frac{e^{-\frac{(x-y)^{2}}{2(t-s)}}}{t-s} |u_{n,k}(s,y)-u_{n,k-1}(s,y) |^{2} dyds\Big ) \leq \\
&\leq c \int_{0}^{t}\sup_{y \in [0,1]} \mathbf{E} \Big (|u_{n,k}(s,y)-u_{n,k-1}(s,y) |^{2} \Big )\int_{0}^{1}\frac{e^{-\frac{(x-y)^{2}}{2(t-s)}}}{t-s}  dyds\leq \\
&\leq c\int_{0}^{t} \sup_{y \in [0,1]} \mathbf{E} \Big (|u_{n,k}(s,y)-u_{n,k-1}(s,y) |^{2}\frac{1}{t-s} \Big)\int_{0}^{1}e^{-\frac{(x-y)^{2}}{2(t-s)}} dy ds\\
&\leq c\int_{0}^{t} \sup_{y \in [0,1]} \mathbf{E} \Big (|u_{n,k}(s,y)-u_{n,k-1}(s,y) |^{2} \Big ) \frac{1}{\sqrt{t-s}} ds \\
& \leq c\int_{0}^{t} P_{n,k}(s) \frac{1}{\sqrt{t-s}} ds \leq C_{T} \left(\int_{0}^{t} P_{n,k}(s)^{q} ds \right)^{\frac{1}{q}},
\end{split}
\end{equation}
where we took exponent $q$ such that $\frac{1}{p}+\frac{1}{q}=1$ for $p \in (1,2)$ so that $\frac{1}{(t-s)^{\frac{p}{2}}}$ is integrable.

We can then go back to $(\ref{p_{n,k}})$ and also get for $R_{n,k}(t)=\sup_{\bar{s}\in [0,t], y \in [0,1]} \left( \mathbf{E} \Big (|u_{n,k}(\bar{s},y)-u_{n,k-1}(\bar{s},y) |^{2}\right)$

\begin{equation} 
\begin{split}
& P_{n,k+1}(t)^{q} \le C_{T}\left(\int_{0}^{t}P_{n,k}(s)^{q} d s\right) \\
& R_{n,k+1}(t)^{q} \le C_{T}\left(\int_{0}^{t}P_{n,k}(s)^{q} d s\right).\\
\end{split}
\end{equation}

By repeatedly applying the above inequality for the integrand in the right-hand side, we get

\begin{equation}
\begin{split}
& R_{n,k+1}(t)^{q} \leq C_{T} \int_{0}^{t} \int_{0}^{s_{k-1}} \cdots \int_{0}^{s_{1}} P_{n,1}(s)^{q}  d s ds_{1}\ldots d s_{k-2} ds_{k-1} \\
&= C_{T} \int_{0}^{t} \int_{0}^{s_{k-1}} \cdots \int_{0}^{s_{1}}  \left( \sup_{y \in [0,1]} \mathbf{E} \Big (|u_{n,1}(s,y)-u_{n,0}(s,y) |^{2} \Big) \right)^{q} d s ds_{1}\ldots d s_{k-2} ds_{k-1}\\
& \leq C_{T} \int_{0}^{t} \int_{0}^{s_{k-1}} \cdots \int_{0}^{s_{1}} \left(\sup_{y \in [0,1] s \in [0,t]} \mathbf{E} \Big (|u_{n,1}(s,y)-u_{n,0}(s,y) |^{2} \Big) \right)^{q} d s ds_{1}\ldots d s_{k-2} ds_{k-1}\\
& \leq (R_{n,1}(T)^{q}) \frac{t^{k}}{k!} .
\end{split}
\end{equation}

\noindent Finally, this gives us that  
$\sum_{k=1}^{\infty}\sup_{\bar{s}\in [0,t], y \in [0,1]} \mathbf{E} \Big (|u_{n,k}(\bar{s},y)-u_{n,k-1}(\bar{s},y) |^{2}\Big)< \infty$ and we can conclude that the sequence $u_{n,k}(t,x)$ is Cauchy in the norm of the space $P_{2,\infty}[0,T]$  given by
\begin{equation} \nonumber
\|\upsilon\|_{P_{2, \infty}}:=\sup_{x \in [0,1], t \in [0,T]}\mathbf{E} \left( |\upsilon(t,x)|^{2} \right).
\end{equation}
Thus, there exists a function $u_n$, such that $u_{n,k} \rightarrow u_n$  in this norm, so that 
\begin{equation} \label{r4r4t53}
\lim_{k \rightarrow \infty} \sup_{t\in [0,T], x\in [0,1]}\mathbf{E} ( |u_{n,k}(t,x)-u_{n}(t,x)|^{2})=0.
\end{equation}
As a consequence, the relation $(\ref{statespotkk})$ gives us
\begin{equation} \label{un}
\begin{split}
u_{n}(t,x) &=\int_{0}^{1} G_t(x,y)u_0(y)dy+\int_{0}^{t}\int_{0}^{1} G_{t-s}(x,y) H_n(|u_{n}(s,y)|)f(s,y,u_{n}(s,y)) dy ds \\
&+ \int_{0}^{t}\int_{0}^{1}G_{t-s}(x,y)H_n(|u_{n}(s,y)|)\sigma(s,y,u_{n}(s,y))W(dy,ds).
\end{split}
\end{equation}
%%%%%%%%%%%%%%UNIQUENESSSSS
%Let's check the uniqueness of the solution u
Let us now prove the uniqueness of the solution $u_n$ to the equation $(\ref{mild0})$. Suppose that $\omega_n$ is another solution of $(\ref{mild0})$, then,
\begin{equation} \nonumber
\begin{split}
&u_n(t,x)-\omega_n(t,x)= \\
&=\int_{0}^{t}\int_{0}^{1} G_{t-s}(x,y) \Big (H_n(|u_n(s,y)|)f(s,y,u_n(s,y)))-H_n(|\omega_n(s,y)|)f(s,y,\omega_n(s,y))) \Big )dy ds \\
&+ \int_{0}^{t}\int_{0}^{1}G_{t-s}(x,y)\Big (H_n(|u_n(s,y)|)\sigma(s,y,u_n(s,y))-H_n(|\omega_n(s,y)|)\sigma(s,y,\omega_n(s,y))\Big )W(dy,ds).
\end{split}
\end{equation}
Following the same steps as in the existence proof when we can extract the equation
\begin{equation} \nonumber
\left(\sup_{x \in [0,1], \bar{s}\in [0,t]}\mathbf{E} \Big (|u_{n}(\bar{s},x)-\omega_{n}(\bar{s},x)|^{2} \Big )\right)^{q} \le C_{T} \int_{0}^{t}  \left(\sup_{x \in [0,1],\bar{s}\in [0,s]}\mathbf{E} \Big ( |u_{n}(\bar{s},x)-\omega_{n}(\bar{s},x)|^{2}  \Big ) \right)^{q} ds. 
\end{equation}
Gronwall's inequality then gives us
\begin{equation} \nonumber
 \sup_{x \in [0,1], \bar{s}\in [0,t]}\mathbf{E} \Big (|u_{n}(\bar{s},x)-\omega_{n}(\bar{s},x)|^{2} \Big)=0
\end{equation}
for any $t \in [0,T]$.
Then, the processes $u_n$ and $\omega_n$ are  equivalent, i.e., they satisfy that 
$$P \Big (\omega \in  \Omega : u_{n}(t,x;\omega)=  \omega_{n}(t,x;\omega) \Big )=1$$ for any $t \in [0,T]$ and any $x \in [0,1]$.
Moreover, considering the continuous versions of $u_n$ and $\omega_n$, they are indistinguishable, since any two processes are indistinguishable in a set if they are equivalent and a.s. continuous in that 
set.
We have now obtained the uniqueness of the solution, $u_n$, of the equation $(\ref{mild0})$, with paths uniquely defined a.s.

%%%%%%%%%%%%%%%%%%%%%%%%%%%%%%%

\section{Malliavin calculus}
In this section we show that $u$ has a local version, meaning that it belongs to the space $L_{1,2}^{loc} \subset D_{1,2}^{loc}$. 
%dddddddddddddddddddd
Therefore, it is necessary to proceed to Malliavin calculus and prove, according to the $L_{1,2}$ definition, that $u_n \in L_{1,2}$. We follow the approach of \cite{AFK},\cite{weber} to extend the regularity of the solution to $({\ref{smr00}})$ to the space $L_{1,2}^{loc}$. The $D_{1,2}$ regularity of $u$ is proven in \cite{N}. First, we must take the Malliavin derivative on the stochastic equation $(\ref{mild0})$. This requires the $D_{1,2}$-regularity of the process $\{u_{n,k}\}_{ k \in \mathbb{N}}$, which is why we proceed by induction. The induction hypothesis at step $k$ is the $D_{1,2}$ regularity of the process $\{u_{n,k}\}_{ k \in \mathbb{N}}$. We will thus prove that the Cauchy sequence, $u_{n,k}$, belongs to the space $D_{1,2}$ for all $(t,x) \in ( [0, T] \times [0,1])$, using induction in combination with the Picard iteration scheme.
\bigskip

We now proceed with the proof of the first part of Theorem $(\ref{theorem21})$. We split the proof into the following four steps.\\

{\it{Step 1: Extraction of Malliavin derivative and the induction hypothesis}}.\\

For $ k=0$, we have the deterministic function of the form $u_{n,0}(t,x)$. Then, the Malliavin derivative yields that $Du_{n,0}=0$, and we trivially have that $u_{n,0} \in D_{1,2}.$ \\
We suppose that $(IH)$ holds for $k \ge 0$ and for any $I \le k$. 
\begin{equation} \label{UD27}
(IH):=
\begin{cases}
 u_{n,I} \in D_{1,2} \mbox{ for every } (t,x) \in ([0,T] \times [0,1]) \\	
\sup_{I \le k} \sup_{t \in [0,T]} \mathbf{E} \bigg ( \int_{0}^{t}\int_{0}^{1} \| D_{s,y}u_{n,I}(t, \cdot)\|^2_{L^{2}([0,1])} dyds \bigg )<\infty \\
\sup_{I \leq k}\sup_{t \in [0,T], x \in [0,1]} \mathbf{E} \bigg ( \int_{0}^{t}\int_{0}^{1} | D_{s,y}u_{n,I}(t, x)|^2 dyds \bigg )<\infty.
\end{cases}
\end{equation}

\noindent It suffices to show that $(IH')$ holds for any $I \le k+1$.  
\begin{equation} \label{UD28}
(IH'):=
\begin{cases}
 u_{n,I} \in D_{1,2} \mbox{ for every } (t,x) \in( [0,T] \times [0,1])	 \\	
\sup_{I \leq k+1}\sup_{t \in [0,T]} \mathbf{E} \bigg ( \int_{0}^{t}\int_{0}^{1} \| D_{s,y}u_{n,I}(t,\cdot)\|^2_{L^{2}([0,1])} dyds \bigg )<\infty \\
\sup_{I \leq k+1}\sup_{t \in [0,T], x \in [0,1]} \mathbf{E} \bigg ( \int_{0}^{t}\int_{0}^{1} | D_{s,y}u_{n,I}(t, x)|^2 dyds \bigg )<\infty.
\end{cases}
\end{equation}

\noindent We will apply the Malliavin derivative to $(\ref{statespotkk})$. We note that by the induction hypothesis, $u_{n, I}(t,x) \in D_{1,2}$ for all $I \leq k$.
Therefore, 

\begin{equation} \label{UD33}
\begin{split}
 D_{s,y}&u_{n,k+1}(t,x) =G_{t-s}(x,y)H_n(|u_{n,k}(s,y)|)\sigma(s,y,u_{n,k}(s,y)) \\
&+\int_{s}^{t}\int_{0}^{1}G_{t-\theta}(x,r)D_{s,y}\Big \{ H_n(|u_{n,k}(\theta,r)|)f(\theta, r,u_{n,k}(\theta,r)) \Big \} drd \theta  \\
&+ \int_{s}^{t}\int_{0}^{1} D_{s,y} \Big \{( H_n(|u_{n,k}(\theta,r)|)\sigma(\theta, r,u_{n,k}(\theta,r))\Big \}G_{t-\theta}(x,r) W(dr,d \theta).
\end{split} 
\end{equation}
As mentioned previously, the Malliavin derivative acts on the time interval $[s,t]$, where $s,t \in [0, T]$. Thus the double-integral of $(\ref{UD27})$, for $s\in [0,t]$, coincides with the same double-interval for $s \in [0,T]$. Then, based on the induction argument $(\ref{UD27})$ we aim to show that for all $k$, $u_{n,k} \in D_{1,2}$ and 
\begin{equation} \label{posnkw3of}
\sup_{t \in [0,T]}\mathbf{E} \bigg ( \int_{0}^{T}\int_{0}^{1} \| D_{s,y}u_{n,k}(t,\cdot)\|^2_{L^{2}([0,1])} dyds \bigg )<\infty.
\end{equation}
 Since $f,\sigma$ are continuously differentiable functions and Lipschitz, we can use the Propositions 1.2.3 and 1.2.4 of \cite{N}. Since $H_n f, H_n \sigma$ are Lipschitz and we have by the induction hypothesis that $u_{n,k} \in D_{1,2}$, Proposition 1.2.4 applies and gives us that there exist random functions $m_2(n,k),\hat{m_2}(n,k)$, bounded by constants $C_{f}$ and $C_{\sigma}$ (that do not depend on $n$) respectively such that 
\begin{equation} \label{pod1}
\begin{split}
D_{s,y}[H_n(|u_{n,k}&(\theta,r)|)f(\theta, r,u_{n,k}(\theta,r))]=m_2(n,k)(\theta,r)D_{s,y}u_{n,k}(\theta,r),
\end{split}
\end{equation}
 $$D_{s,y}[H_n(|u_{n,k}(\theta,r)|)\sigma(\theta, r,u_{n,k}(\theta,r))]=\hat{m}_2(n,k)(\theta,r)D_{s,y}u_{n,k}(\theta,r),$$ with $|m_2(n,k)(t,x)| \le C_{f}$, $|\hat{m}_2(n,k)(t,x)| \le C_{\sigma}$ for any $t \in [0,T]$ and $x \in [0,1]$. Replacing the Malliavin derivative in $(\ref{UD33})$ gives us,
\begin{equation} \label{OoOf}
\begin{split}
 D_{s,y}&u_{n,k+1}(t,x) =G_{t-s}(x,y)H_n(|u_{n,k}(s,y)|)\sigma(s,y,u_{n,k}(s,y)) \\
&+\int_{s}^{t}\int_{0}^{1}G_{t-\theta}(x,r)m_2(n,k)(r,\theta)D_{s,y}u_{n,k}(\theta,r) drd \theta  \\
&+ \int_{s}^{t}\int_{0}^{1} G_{t-\theta}(x,r) \hat{m}_2(n,k)(r,\theta)D_{s,y}u_{n,k}(\theta,r) W(dr,d \theta),
\end{split} 
\end{equation}
and $D_{s,y}u_{n,k+1}(t,x)=0$ for $s>t$.\\

{\it{Step 2: The Cauchy sequence $u_{n,k}$ belongs to the space $D_{1,2}$ for all $k$.}} \\

For all $x \in [0,1]$, $I \leq k$ we have that since $|H_{n}(u)\sigma(s,y,u)|$ is bounded,
\begin{equation}  \nonumber
\begin{split}
|D_{s,y}&u_{n,I+1}(t,x)|^{2} \le c \Big [\int_{s}^{t}\int_{0}^{1}|G_{t-\theta}(x,r)||m_2(n,k)(\theta,r)||D_{s,y}u_{n,I}(\theta,r)| dr d \theta  \Big ]^{2} \\
&+c \Big |\int_{s}^{t}\int_{0}^{1}G_{t-\theta}(x,r)\hat{m}_2(n,k)(\theta,r)D_{s,y}u_{n,I}(\theta,r) W(dr,d \theta)  \Big |^{2}\\
&+c |G_{t-s}(x,y)|^{2}
\end{split} 
\end{equation}
for a constant $c>0$. Then since $m_2$ and $\hat{m}_2$ are bounded we get,
\begin{equation}  \label{blue1}
\begin{split}
\mathbf{E} \bigg (\int_{0}^{t} \int_{0}^{1}&  \|D_{s,y}u_{n,I+1}(t,\cdot)\|_{L^{2}([0,1])} ^{2} dyds \bigg )\\
& \le c\mathbf{E} \bigg (\int_{0}^{t} \int_{0}^{1}\int_{0}^{1} \Big [\int_{s}^{t}\int_{0}^{1}|G_{t-\theta}(x,r)||D_{s,y}u_{n,I}(\theta,r)| dr d \theta \Big ]^{2} dx dyds \bigg )\\
&+c\mathbf{E} \bigg (\int_{0}^{t} \int_{0}^{1} \int_{0}^{1}  \Big |\int_{s}^{t}\int_{0}^{1}G_{t-\theta}(x,r)D_{s,y}u_{n,I}(\theta,r) W(dr,d \theta)  \Big |^{2} dx dyds \bigg )\\
&+c\mathbf{E} \bigg (\int_{0}^{t} \int_{0}^{1}\int_{0}^{1}  |G_{t-s}(x,y)|^{2} dx dy ds \bigg ) \\
&:=B_1(t)+B_2(t)+B_3(t). 
\end{split} 
\end{equation}

\noindent We proceed as follows to bound the term $B_{1}(t)$.
\begin{equation}
\begin{split}
&\mathbf{E} \left(\Big [\int_{s}^{t}\int_{0}^{1}|G_{t-\theta}(x,r)||D_{s,y}u_{n,I}(\theta,r)| dr d \theta \Big ]^{2} \right) \leq \mathbf{E} \bigg( \Big [\int_{s}^{t}\int_{0}^{1} \frac{e^{\frac{-(x-r)^{2}}{4\pi (t-\theta)}}}{\sqrt{t-\theta}}|D_{s,y}u_{n,I}(\theta,r)| dr d \theta \Big ]^{2} \bigg)\leq \\
& \mathbf{E} \bigg(\Big [\int_{s}^{t} \frac{1}{\sqrt{t-\theta}} \Big(\int_{0}^{1} e^{\frac{-(x-r)^{2}}{2\pi (t-\theta)}} dr \Big)^{\frac{1}{2}} \Big(\int_{0}^{1}|D_{s,y}u_{n,I}(\theta,r)|^{2} dr\Big)^{\frac{1}{2}} d \theta \Big ]^{2} \bigg)\leq c \mathbf{E} \bigg(\Big [\int_{s}^{t} \frac{1}{(t-\theta)^{\frac{1}{4}}} \Big(\int_{0}^{1}|D_{s,y}u_{n,I}(\theta,r)|^{2} dr\Big)^{\frac{1}{2}} d \theta \Big ]^{2} \bigg) \\
& \leq c \mathbf{E} \bigg( \int_{s}^{t} \frac{1}{(t-\theta)^{\frac{1}{2}}} d \theta \int_{s}^{t} \Big(\int_{0}^{1}|D_{s,y}u_{n,I}(\theta,r)|^{2} dr\Big) d \theta \bigg) \leq C_{T} \mathbf{E} \bigg(\int_{s}^{t} \int_{0}^{1}|D_{s,y}u_{n,I}(\theta,r)|^{2} dr d \theta \bigg) \\
&= C_{T} \mathbf{E} \bigg(\int_{s}^{t} \|D_{s,y}u_{n,I}(\theta, \cdot)\|_{L^2([0,1])}^{2} d \theta \bigg),\\
\end{split}
\end{equation}

\noindent so we have that
\small
\begin{equation} \nonumber
\begin{split}
& \mathbf{E} \bigg (\int_{0}^{t} \int_{0}^{1}\int_{0}^{1} \Big [\int_{s}^{t}\int_{0}^{1}|G_{t-\theta}(x,r)||D_{s,y}u_{n,I}(\theta,r)| dr d \theta \Big ]^{2} dx dyds \bigg ) \leq \mathbf{E} \bigg (\int_{0}^{t} \int_{0}^{1} C_{T} \int_{s}^{t} \|D_{s,y}u_{n,I}(\theta,\cdot)\|_{L^2([0,1])}^{2} d \theta dyds \bigg ) \\
& \leq C_{T}  \int_{0}^{t}  \mathbf{E} \bigg (\int_{0}^{\theta} \int_{0}^{1} \|D_{s,y}u_{n,I}(\theta,\cdot)\|_{L^2([0,1])}^{2} dy ds \bigg )  d \theta 
\end{split}
\end{equation}
\normalsize

\noindent and we can conclude that

\small
\begin{equation} \nonumber
\begin{split}
& B_1(t) \leq  C_{T}  \int_{0}^{t}  \mathbf{E} \bigg (\int_{0}^{\theta} \int_{0}^{1} \|D_{s,y}u_{n,I}(\theta,\cdot)\|_{L^2([0,1])}^{2} dy ds \bigg )  d \theta.
\end{split}
\end{equation}
\normalsize

%BBBBBBBBBBBBBBBBBBBBBB
We now bound the $B_{2}$ term,
\begin{equation} \nonumber
\begin{split}
& B_{2}(t)=\mathbf{E} \bigg (\int_{0}^{t} \int_{0}^{1} \int_{0}^{1}  \Big |\int_{s}^{t}\int_{0}^{1}G_{t-\theta}(x,r)D_{s,y}u_{n,I}(\theta,r) W(dr,d \theta)  \Big |^{2} dx dyds \bigg ) \\
 &\le c\int_{0}^{t} \int_{0}^{1} \int_{0}^{1} \mathbf{E} \bigg ( \int_{s}^{t}\int_{0}^{1}|G_{t-\theta}(x,r)|^2|D_{s,y}u_{n,I}(\theta,r)|^2 drd \theta \bigg ) dx dyds . \\
 &= c \mathbf{E} \bigg ( \int_{0}^{t} \int_{0}^{1} \int_{0}^{1} \int_{s}^{t}\int_{0}^{1}|G_{t-\theta}(x,r)|^2|D_{s,y}u_{n,I}(\theta,r)|^2 drd \theta  dx dyds \bigg ) \\
 &= c \mathbf{E} \bigg ( \int_{0}^{t} \int_{0}^{1}  \int_{s}^{t}\int_{0}^{1} \int_{0}^{1}|G_{t-\theta}(x,r)|^2 dx |D_{s,y}u_{n,I}(\theta,r)|^2 drd \theta  dyds \bigg ) \\
 & \leq c \mathbf{E} \bigg ( \int_{0}^{t} \int_{0}^{1}  \int_{s}^{t}\int_{0}^{1} \frac{1}{\sqrt{t-\theta}} |D_{s,y}u_{n,I}(\theta,r)|^2 drd \theta  dyds \bigg ) \\
 & \leq c  \int_{0}^{t} \frac{1}{\sqrt{t-\theta}} \mathbf{E} \bigg (   \int_{0}^{\theta} \int_{0}^{1} \int_{0}^{1}  |D_{s,y}u_{n,I}(\theta, r)|^2 dr dy ds \bigg ) d \theta \\
 & \leq c  \int_{0}^{t} \frac{1}{\sqrt{t-\theta}} \mathbf{E} \bigg (   \int_{0}^{\theta} \int_{0}^{1} \| D_{s,y}u_{n,I}(\theta, \cdot)\|^{2}_{L^2([0,1])} dy ds \bigg ) d \theta.
\end{split}
\end{equation}

We conclude that 
\begin{equation}
\begin{split}
& B_{2}(t) \leq c  \int_{0}^{t} \frac{1}{\sqrt{t-\theta}} \mathbf{E} \bigg (   \int_{0}^{\theta} \int_{0}^{1} \| D_{s,y}u_{n,I}(\theta, \cdot)\|^{2}_{L^2([0,1])} dy ds \bigg ) d \theta.
\end{split}
\end{equation}

For the term $B_3$ we have that
\begin{equation}
\begin{split}
& B_{3}(t) = \mathbf{E} \bigg (\int_{0}^{t} \int_{0}^{1}\int_{0}^{1} |G_{t-s}(x,y)|^{2} dx dy ds \bigg ) \leq C_{T}. \\
\end{split}
\end{equation}

Summarizing the terms $B_i$ for $i=1,\cdots,3$ in $(\ref{blue1})$ we get,
\begin{equation}  \label{Ox12}
\begin{split}
\mathbf{E} \bigg (\int_{0}^{t} \int_{0}^{1}&  \|D_{s,y}u_{n,I+1}(t,\cdot)\|^2_{L^{2}([0,1])} dyds \bigg ) \leq \\
&\leq C_{T}\int_{0}^{t} \mathbf{E} \bigg (\int_{0}^{\theta}\int_{0}^{1} \|D_{s,y}u_{n,I}(\theta, \cdot)\|^2_{L^{2}([0,1])}   dyds \bigg )d \theta \\
&+c\int_{0}^{t} \frac{1}{\sqrt{t-\theta}} \mathbf{E} \bigg (\int_{0}^{\theta} \int_{0}^{1}\|D_{s,y}u_{n,I}(\theta, \cdot)\|_{L^{2}([0,1])}^2  dyds \bigg )d\theta+C_{T}.
\end{split}
\end{equation}
Taking supremum on $I \le k$, and taking into account that the constants are independent of $I$, we get
\begin{equation}  \nonumber
\begin{split}
\sup_{I \le k}\mathbf{E} \bigg (\int_{0}^{t} \int_{0}^{1}&  \|D_{s,y}u_{n,I+1}(t,\cdot)\|^2_{L^{2}([0,1])} dyds \bigg )\\
&\leq C_{T}+c \int_{0}^{t} \frac{1}{\sqrt{t-\theta}}\sup_{I \le k}\mathbf{E} \Big [\int_{0}^{\theta} \int_{0}^{1} \|D_{s,y}u_{n,I}(\theta,\cdot)\|^2_{L^{2}([0,1])}  dyds \Big ]d \theta  \\
&+C_{T} \int_{0}^{t}\sup_{I \le k}\mathbf{E} \Big [\int_{0}^{\theta} \int_{0}^{1} \|D_{s,y}u_{n,I}(\theta,\cdot)\|^2_{L^{2}([0,1])}  dyds \Big ]d \theta.
\end{split} 
\end{equation}
We will need $D_{s,y}u_{n, I+1}$ in the right-hand side of the previous inequality to apply Gronwall's inequality eventually. 
If we set 
$$Z(\theta;I):=\sup_{I \le k}\mathbf{E} \Big [\int_{0}^{\theta} \int_{0}^{1} \|D_{s,y}u_{n,I}(\theta,\cdot)\|^2_{L^{2}([0,1])}  dy ds \Big ],$$
then we have that
\begin{equation}  \label{dd401kmbjhu711}
\begin{split}
\sup_{I \le k}\mathbf{E} \bigg (\int_{0}^{t} \int_{0}^{1}&  \|D_{s,y}u_{n,I+1}(t,\cdot)\|^{2}_{L^{2}([0,1])} dyds \bigg )\\
&\leq   C_{T}+c \int_{0}^{t} \frac{1}{\sqrt{t-\theta}}Z(\theta;I)d \theta +C_{T} \int_{0}^{t}Z(\theta;I)d \theta. 
\end{split} 
\end{equation}
In addition,
\begin{equation}  \nonumber
\begin{split}
Z(\theta;I):=\sup_{I \le k}\mathbf{E} \Big [\int_{0}^{\theta} \int_{0}^{1} \Big [\|D_{s,y}u_{n,I}(\theta, \cdot)-D_{s,y}u_{n,I+1}(\theta,\cdot)+D_{s,y}u_{n,I+1}(\theta,\cdot)\|_{L^{2}([0,1])}\Big ]^2  dyds \Big ].
\end{split} 
\end{equation}
Minkowski's inequality gives us
\begin{equation}  \nonumber
\begin{split}
Z(\theta;I) \le \sup_{I \le k}\mathbf{E} \Big [\int_{0}^{\theta} \int_{0}^{1} \Big [\|D_{s,y}u_{n,I}(\theta,\cdot)-D_{s,y}u_{n,I+1}(\theta,\cdot)\|_{L^{2}([0,1])}+\|D_{s,y}u_{n,I+1}(\theta,\cdot)\|_{L^{2}([0,1])}\Big ]^2  dyds \Big ].
\end{split} 
\end{equation}
Let $$\zeta(\theta;I):=\Big [\|D_{s,y}u_{n,I}(\theta,\cdot)-D_{s,y}u_{n,I+1}(\theta,\cdot)\|_{L^{2}([0,1])}+\|D_{s,y}u_{n,I+1}(\theta,\cdot)\|_{L^{2}([0,1])}\Big ]^2.$$
Then
\begin{equation}  \nonumber
\begin{split}
Z(\theta;I) \le \sup_{I \le k}\mathbf{E} \Big [\int_{0}^{\theta} \int_{0}^{1} \zeta(\theta;I)  dyds \Big ].
\end{split} 
\end{equation}
Some basic calculations on $\zeta(\theta;I)$ give us,
\begin{equation}  \nonumber
\begin{split}
\zeta(\theta;I)
&:=\|D_{s,y}u_{n,I}(\theta,\cdot)\|^2_{L^{2}([0,1])}+2\|D_{s,y}u_{n,I+1}(\theta,\cdot)\|^2_{L^{2}([0,1])} \\
&+2\|D_{s,y}u_{n,I}(\theta,\cdot)\|_{L^{2}([0,1])}\|D_{s,y}u_{n,I+1}(\theta,\cdot)\|_{L^{2}([0,1])} \\
&+2\|D_{s,y}u_{n,I}(\theta,\cdot)-D_{s,y}u_{n,I+1}(\theta, \cdot)\|_{L^{2}([0,1])}\|D_{s,y}u_{n,I+1}(\theta,\cdot)\|_{L^{2}([0,1])}.
\end{split} 
\end{equation}
Minkowski's inequality yields
\begin{equation}  \nonumber
\begin{split}
\zeta(\theta;I)&\le \|D_{s,y}u_{n,I}(\theta,\cdot)\|^2_{L^{2}([0,1])}+4\|D_{s,y}u_{n,I+1}(\theta,\cdot)\|^2_{L^{2}([0,1])} \\
&+4\|D_{s,y}u_{n,I}(\theta, \cdot)\|_{L^{2}([0,1])}\|D_{s,y}u_{n,I+1}(\theta,\cdot)\|_{L^{2}([0,1])}.
\end{split} 
\end{equation}
Then,
\begin{equation}  \nonumber
\begin{split}
Z(\theta;I) &\le \sup_{I \le k}\mathbf{E} \Big [\int_{0}^{\theta} \int_{0}^{1} \|D_{s,y}u_{n,I}(\theta,\cdot)\|^2_{L^{2}([0,1])}  dyds \Big ] \\
&+ 4\sup_{I \le k}\mathbf{E} \Big [\int_{0}^{\theta} \int_{0}^{1} \|D_{s,y}u_{n,I+1}(\theta,\cdot)\|^2_{L^{2}([0,1])}  dyds \Big ] \\
&+4 \sup_{I \le k}\mathbf{E} \Big [\int_{0}^{\theta} \int_{0}^{1}\|D_{s,y}u_{n,I}(\theta,\cdot)\|_{L^{2}([0,1])}\|D_{s,y}u_{n,I+1}(\theta,\cdot)\|_{L^{2}([0,1])}  dyds \Big ]. 
\end{split} 
\end{equation}
%OOOOOOOOOOOOOOOOOOOOOOOOO
Since the Malliavin derivative is zero for $s \in (\theta,t)$, we have that,
\begin{equation}  \nonumber
\begin{split}
Z(\theta;I) &\le \sup_{I \le k}\mathbf{E} \Big [\int_{0}^{t} \int_{0}^{1} \|D_{s,y}u_{n,I}(\theta,\cdot)\|^2_{L^{2}([0,1])}  dyds \Big ] \\
&+ 4\sup_{I \le k}\mathbf{E} \Big [\int_{0}^{t} \int_{0}^{1} \|D_{s,y}u_{n,I+1}(\theta,\cdot)\|^2_{L^{2}([0,1])}  dyds \Big ] \\
&+4 \sup_{I \le k}\mathbf{E} \Big \{\int_{0}^{t} \int_{0}^{1}\frac{1}{2}\Big [\|D_{s,y}u_{n,I}(\theta,\cdot)\|^2_{L^{2}([0,1])}+\|D_{s,y}u_{n,I+1}(\theta,\cdot)\|^2_{L^{2}([0,1])} \Big ] dyds \Big \}.
\end{split} 
\end{equation}
Bounding some terms by the supremum with respect to $t \in [0,T]$ we get
\begin{equation}  \nonumber
\begin{split}
Z(\theta;I)
&\le  \sup_{t \in [0,T]}\sup_{I \le k}\mathbf{E} \Big [\int_{0}^{t} \int_{0}^{1} \|D_{s,y}u_{n,I}(\theta,\cdot)\|^2_{L^{2}([0,1])}  dy ds \Big ] \\
&+ 4\sup_{I \le k}\mathbf{E} \Big [\int_{0}^{t} \int_{0}^{1} \|D_{s,y}u_{n,I+1}(\theta,\cdot)\|^2_{L^{2}([0,1])}  dyds \Big ] \\
&+2  \sup_{t \in [0,T]}\sup_{I \le k}\mathbf{E} \Big \{\int_{0}^{t} \int_{0}^{1}\|D_{s,y}u_{n,I}(\theta,\cdot)\|^2_{L^{2}([0,1])}dyds \Big \} \\
&+2 \sup_{I \le k}\mathbf{E} \Big \{\int_{0}^{t} \int_{0}^{1}\|D_{s,y}u_{n,I+1}(\theta,\cdot)\|^2_{L^{2}([0,1])} dyds \Big \}.
\end{split} 
\end{equation}
Using the induction hypothesis in $(\ref{UD27})$ for $I \leq k$, we have
\begin{equation}  \nonumber
\begin{split}
Z(\theta;I) &\le C_{T}
+6\sup_{I \le k}\mathbf{E} \Big \{\int_{0}^{t} \int_{0}^{1}\|D_{s,y}u_{n,I+1}(\theta,\cdot)\|^2_{L^{2}([0,1])}dyds \Big \} 
\end{split} 
\end{equation}
for the positive constant $C_{T}$. Upper bounding the quantity $Z(\theta;I)$ in the relation $(\ref{dd401kmbjhu711})$, yields
\begin{equation}  \nonumber
\begin{split}
\sup_{I \le k}\mathbf{E} \bigg (\int_{0}^{t} \int_{0}^{1}&  \|D_{s,y}u_{n,I+1}(t,\cdot)\|^{2}_{L^{2}([0,1])} dyds \bigg )\\
&\le  C_{T}+c \int_{0}^{t} \frac{1}{\sqrt{t-\theta}}\Big[C_{T}
+6 \sup_{I \le k}\mathbf{E} \Big \{\int_{0}^{\theta} \int_{0}^{1}\|D_{s,y}u_{n,I+1}(\theta,\cdot)\|^2_{L^{2}([0,1])} dyds \Big \}\Big ]d \theta  \\
&+C_{T} \int_{0}^{t}\Big [C_{T}
+6 \sup_{I \le k}\mathbf{E} \Big \{\int_{0}^{\theta} \int_{0}^{1}\|D_{s,y}u_{n,I+1}(\theta,\cdot)\|^2_{L^{2}([0,1])} dyds \Big \}\Big]d \theta.
\end{split} 
\end{equation}
Equivalently,
\begin{equation}  \nonumber \label{an,k}
\begin{split}
\sup_{I \le k}\mathbf{E} \bigg (\int_{0}^{t} \int_{0}^{1}&  \|D_{s,y}u_{n,I+1}(t,\cdot)\|^{2}_{L^{2}([0,1])} dyds \bigg )\\
&\le C_{T}+c \int_{0}^{t} \frac{1}{\sqrt{t-\theta}}
 \sup_{I \le k}\mathbf{E} \Big \{\int_{0}^{\theta} \int_{0}^{1}\|D_{s,y}u_{n,I+1}(\theta,\cdot)\|^2_{L^{2}([0,1])} dyds \Big \}d \theta  \\
&+C_{T} \int_{0}^{t}
\sup_{I \le k}\mathbf{E} \Big \{\int_{0}^{\theta} \int_{0}^{1}\|D_{s,y}u_{n,I+1}(\theta,\cdot)\|^2_{L^{2}([0,1])} dyds \Big \}d \theta.
\end{split} 
\end{equation}
If we set  
\begin{equation} \label{df443ff4f5f6d}
A_{n,k}(t):=\sup_{I \le k}\mathbf{E} \bigg (\int_{0}^{t} \int_{0}^{1}  \|D_{s,y}u_{n,I}(t,\cdot)\|^2_{L^{2}([0,1])} dyds \bigg )
\end{equation}
then
\begin{equation}  \label{df443ffffd}
A_{n,k+1}(t) \le C_{T}+c \int_{0}^{t} \frac{1}{\sqrt{t-\theta}}A_{n,k+1}( \theta) d \theta+C_{T} \int_{0}^{t}A_{n,k+1}( \theta)d \theta.
\end{equation}

 Using this equation, we can get that
 \small
\begin{equation}
\begin{split}
&\int_{0}^{t} \frac{1}{\sqrt{t-\tilde{s}}} A_{n, k+1}(\tilde{s}) d \tilde{s} \leq C_{T}\int_{0}^{t} \frac{1}{\sqrt{t-\tilde{s}}} d \tilde{s}+c \int_{0}^{t} \int_{0}^{\tilde{s}} \frac{1}{\sqrt{t-\tilde{s}}} \frac{1}{\sqrt{\tilde{s}-\theta}} A_{n, k+1}(\theta) d \theta d \tilde{s} + C_{T}\int_{0}^{t} \frac{1}{\sqrt{t-\tilde{s}}} \int_{0}^{\tilde{s}} A_{n, k+1}(\theta) d \theta d\tilde{s} \\
& \leq C_{T}+c \int_{0}^{t}\left(\int_{\theta}^{t}\frac{1}{\sqrt{t-\tilde{s}}} \frac{1}{\sqrt{\tilde{s}-\theta}} d \tilde{s}\right) A_{n, k+1}(\theta) d \theta + C_{T}\int_{0}^{t} \frac{1}{\sqrt{t-\tilde{s}}} d\tilde{s} \int_{0}^{t} A_{n, k+1}(\theta) d \theta \\
& \leq C_{1,T}+C_{2,T} \int_{0}^{t} A_{n, k+1}(\tau) d \tau, \\
\end{split}
\end{equation}

where we used that 
$$
\left(\int_{\theta}^{t}\frac{1}{\sqrt{t-\tilde{s}}} \frac{1}{\sqrt{\tilde{s}-\theta}} d \tilde{s}\right) < \infty, \int_{0}^{t} \frac{1}{\sqrt{t-\tilde{s}}} d\tilde{s}< \infty.
$$

As a result we can go back to $(\ref{df443ffffd})$ and get

\begin{equation} 
A_{n,k+1}(t) \le C_{1,T}+C_{2,T}\int_{0}^{t}A_{n,k+1}( \theta)d \theta.
\end{equation}

Finally, Gronwall's lemma  yields
\begin{equation} \label{f11hhhh122222hsh}
A_{n,k+1}(t):=\sup_{I \le k+1}\mathbf{E} \bigg (\int_{0}^{t} \int_{0}^{1}  \|D_{s,y}u_{n,I}(t,\cdot)\|^2_{L^{2}([0,1])} dyds \bigg )\le C_{T}.
\end{equation}
Taking supremum in $t \in [0,T]$
\begin{equation} \label{l2norm}
\sup_{t \in [0,T]}\sup_{I \le k+1}\mathbf{E} \bigg (\int_{0}^{t} \int_{0}^{1} \|D_{s,y}u_{n,I}(t,\cdot)\|^{2}_{L^{2}([0,1])} dyds \bigg )\le C_{T},
\end{equation}
where the positive constant $C_{T}$ is independent of $t$ and $k$, and only depends on the final time $T$. 
In the case $s>t$ we have that $D_{s,y}u_{n,I}(t,x)=0$. 
\newline
We now prove that  $u_{n,k+1} \in D_{1,2}$.
In accordance with $(\ref{fokfogiiggot9g})$ and $(\ref{norm3})$,
\begin{equation} \label{normtokf004}
\|u_{n,k+1}(t,x)\|_{D_{1,2}}
\le \Big [ \mathbf{E} (|u_{n,k+1}(t,x)|^2) +\mathbf{E} \bigg (\int_{0}^{T} \int_{0}^{1} |D_{s,y}u_{n,k+1}(t,x)|^2dyds \bigg ) \Big ]^{1/2}. 
\end{equation}

\noindent We know that $\sup_{t\in [0,T], x\in [0,1]}\mathbf{E} ( |u_{n,k}(t,x)|^{2} ) < C_{T}$, since we showed that $u_{n,k}$ is Cauchy in $P_{2,\infty}$. Thus
\begin{equation} \nonumber
\begin{split}
\|u_{n,k+1}(t,x)\|_{D_{1,2}}& \le\Big [ C_{T} +\mathbf{E} \bigg (\int_{0}^{T} \int_{0}^{1} |D_{s,y}u_{n,k+1}(t,x)|^2dyds \bigg ) \Big ]^{1/2}.  \\
\end{split}
\end{equation}

It suffices to show that 
\begin{equation} \nonumber
\begin{split}
\mathbf{E} \bigg (\int_{0}^{t} \int_{0}^{1} |D_{s,y}u_{n,k+1}(t,x)|^2dyds \bigg ) < \infty.  \\
\end{split}
\end{equation}

We have that for all $x \in [0,1], t \in [0,T]$,
\begin{equation}  \label{omg}
\begin{split}
\mathbf{E} \bigg(\int_{0}^{t}\int_{0}^{1}|D_{s,y}&u_{n,k+1}(t,x)|^{2} dy ds \bigg) \le c  \int_{0}^{t} \int_{0}^{1} \mathbf{E} \bigg(\Big [\int_{s}^{t}\int_{0}^{1}|G_{t-\theta}(x,r)||D_{s,y}u_{n,k}(\theta,r)| drd \theta  \Big ]^{2} \bigg) dy ds \\
&+c \mathbf{E} \bigg(\int_{0}^{t} \int_{0}^{1} \Big |\int_{s}^{t}\int_{0}^{1}G_{t-\theta}(x,r) D_{s,y}u_{n,k}(\theta,r) W(dr,d \theta)  \Big |^{2} dy ds \bigg)\\
&+c \sup_{x \in [0,1]}\int_{0}^{t} \int_{0}^{1} |G_{t-s}(x,y)|^{2} dy ds \\
&: = \Gamma_{1}(t)+\Gamma_{2}(t)+\Gamma_{3}(t).
\end{split} 
\end{equation}

\noindent We can bound the term $\Gamma_{3}(t)$ as we did for $B_{3}$ by a constant $C_{T}$.

\noindent For the term $\Gamma_{1}(t)$ we proceed as follows
\begin{equation} \nonumber
\begin{split}
&\mathbf{E} \left(\Big [\int_{s}^{t}\int_{0}^{1}|G_{t-\theta}(x,r)||D_{s,y}u_{n,k}(\theta,r)| dr d \theta \Big ]^{2} \right) \leq \mathbf{E} \bigg( \Big [\int_{s}^{t}\int_{0}^{1} \frac{e^{\frac{-(x-r)^{2}}{4\pi (t-\theta)}}}{\sqrt{t-\theta}}|D_{s,y}u_{n,k}(\theta,r)| dr d \theta \Big ]^{2} \bigg)\leq \\
& \mathbf{E} \bigg(\Big [\int_{s}^{t} \frac{1}{\sqrt{t-\theta}} \Big(\int_{0}^{1} e^{\frac{-(x-r)^{2}}{2\pi (t-\theta)}} dr \Big)^{\frac{1}{2}} \Big(\int_{0}^{1}|D_{s,y}u_{n,k}(\theta,r)|^{2} dr\Big)^{\frac{1}{2}} d \theta \Big ]^{2} \bigg)\leq c \mathbf{E} \bigg(\Big [\int_{s}^{t} \frac{1}{(t-\theta)^{\frac{1}{4}}} \Big(\int_{0}^{1}|D_{s,y} u_{n,k}(\theta,r)|^{2} dr\Big)^{\frac{1}{2}} d \theta \Big ]^{2} \bigg) \\
& \leq c \mathbf{E} \bigg( \int_{s}^{t} \frac{1}{(t-\theta)^{\frac{1}{2}}} d \theta \int_{s}^{t} \Big(\int_{0}^{1}|D_{s,y}u_{n,k}(\theta,r)|^{2} dr\Big) d \theta \bigg) \leq C_{T} \mathbf{E} \bigg(\int_{s}^{t} \Big(\int_{0}^{1}|D_{s,y}u_{n,k}(\theta,r)|^{2} dr\Big) d \theta \bigg) \\
&= C_{T} \mathbf{E} \bigg(\int_{s}^{t} \|D_{s,y}u_{n,k}(\theta,\cdot)\|_{L^2([0,1])}^{2} d \theta \bigg),\\
\end{split}
\end{equation}

\noindent so we have that
\small
\begin{equation} \nonumber
\begin{split}
&\mathbf{E} \bigg (\int_{0}^{t} \int_{0}^{1}\int_{0}^{1} \Big [\int_{s}^{t}\int_{0}^{1}|G_{t-\theta}(x,r)||D_{s,y}u_{n,k}(\theta,r)| dr d \theta \Big ]^{2} dx dyds \bigg ) \leq \mathbf{E} \bigg (\int_{0}^{t} \int_{0}^{1} C_{T} \int_{s}^{t} \|D_{s,y}u_{n,k}(\theta,\cdot)\|_{L^2([0,1])}^{2} d \theta dyds \bigg ) \\
& \leq C_{T}  \int_{0}^{t}  \mathbf{E} \bigg (\int_{0}^{\theta} \int_{0}^{1} \|D_{s,y}u_{n,k}(\theta,\cdot)\|_{L^2([0,1])}^{2} dy ds \bigg )  d \theta \leq C_{T}  \sup_{ \theta \in [0,T]}  \mathbf{E} \bigg (\int_{0}^{\theta} \int_{0}^{1} \|D_{s,y}u_{n,k}(\theta,\cdot)\|_{L^2([0,1])}^{2} dy ds \bigg ) \leq C_{T},
\end{split}
\end{equation}
\normalsize

\noindent where we used ($\ref{l2norm}$).

\noindent It remains to bound $\Gamma_{2}(t)$.
We have supposed in the induction hypothesis that 
$$ \sup_{t \in [0,T], x \in [0,1]} \mathbf{E} \bigg ( \int_{0}^{t}\int_{0}^{1} | D_{s,y}u_{n,I}(t, x)|^2 dyds \bigg ):=C_{T}<\infty$$ for all $I \leq k$ and our aim now is to show that it also holds for $k+1$. We have

\begin{equation}  \nonumber \label{gamma_2bound}
\begin{split}
& \mathbf{E} \bigg(\int_{0}^{t} \int_{0}^{1} \Big |\int_{s}^{t}\int_{0}^{1}G_{t-\theta}(x,r)\hat{m}_2(n,k)(\theta,r)D_{s,y}u_{n,k}(\theta,r) W(dr,d \theta)  \Big |^{2} dy ds \bigg) \\
&\leq c \mathbf{E} \bigg( \int_{0}^{t} \int_{0}^{1}  \int_{s}^{t} \int_{0}^{1}G^{2}_{t-\theta}(x,r)|D_{s,y}u_{n,k}(\theta,r)|^{2} dr d\theta dy ds \bigg)\\
& \leq c\int_{0}^{t} \frac{1}{\sqrt{t-\theta}} \sup_{r \in [0,1]} \mathbf{E} \bigg(\int_{0}^{\theta} \int_{0}^{1} |D_{s,y}u_{n,k}(\theta,r)|^{2} dy ds \bigg) d\theta \\
\end{split} 
\end{equation}

\noindent and this implies in particular that 
$\Gamma_{2}(t) \leq C_{k,T}$ using the third condition of the induction for $k$. \\
Combining these estimates we can conclude that $ u_{n,k+1}(t,x) \in D_{1,2}$. \\

We can also conclude that $$\sup_{t \in [0,T], x \in [0,1]} \mathbf{E} \bigg ( \int_{0}^{t}\int_{0}^{1} | D_{s,y}u_{n,k+1}(t, x)|^2 dyds \bigg )<\infty.$$

\noindent since we bounded all the terms by constants that only depend on the final time $T$ and $k$. 

\noindent Moreover, the terms $\Gamma_{1},\Gamma_{3}$ have been bounded by constants that only depend on the final time $T$ and not on $k$.
Combining with what we showed for $\Gamma_{2}(t)$ we have the equation:

\begin{equation}  \label{Ox12_2}
\begin{split}
\sup_{x \in [0,1]}\mathbf{E} \bigg (\int_{0}^{t} \int_{0}^{1}&  |D_{s,y}u_{n,k+1}(t,x)|^2 dyds \bigg ) \leq C_{1,T}\\
&+C_{2,T}\int_{0}^{t} \frac{1}{\sqrt{t-\theta}}\sup_{r \in [0,1]} \mathbf{E} \bigg (\int_{0}^{\theta} \int_{0}^{1}|D_{s,y}u_{n,k}(\theta,r)|^{2}  dyds \bigg )d\theta.
\end{split}
\end{equation}

\noindent So following exactly what we did in Step 2, applied to the equation $(\ref{Ox12_2})$ this time, we can get a bound that does not depend on $k$, so that

\begin{equation} \label{f11hhhh22222hsh}
\sup_{I \le k+1}\sup_{t \in [0,T], x \in [0,1]}\mathbf{E} \bigg (\int_{0}^{t} \int_{0}^{1} |D_{s,y}u_{n,I}(t,x)|^{2} dyds \bigg )\le C_{T}.
\end{equation}

\noindent The induction is complete, and thus $(\ref{UD27})$ holds, i.e.,
\begin{equation} \nonumber
\forall k \text{ }u_{n,k}(t,x) \in D_{1,2} \mbox{ for every } (t,x) \in [0,T] \times [0,1]
\end{equation}
and
\begin{equation} \nonumber
\forall k \text{ } \sup_{I \leq k}\sup_{t \in [0,T]} \mathbf{E} \bigg ( \int_{0}^{t}\int_{0}^{1} \| D_{s,y}u_{n,I}(t,\cdot)\|^2_{L^{2}([0,1])} dyds \bigg )\leq C_{T}
\end{equation}
and
\begin{equation} \nonumber
\forall k \text{ } \sup_{I \leq k} \sup_{t \in [0,T], x \in [0,1]} \mathbf{E} \bigg ( \int_{0}^{t}\int_{0}^{1} | D_{s,y}u_{n,I}(t, x)|^2 dyds \bigg )\leq C_{T}.
\end{equation}
\newline

\begin{remark} \label{iter}
We have that 
\begin{equation}  \label{Ox12_copy}
\begin{split}
\sup_{x \in [0,1]}\mathbf{E} \bigg (\int_{0}^{t} \int_{0}^{1}&  |D_{s,y}u_{n,k+1}(t,x)|^2 dyds \bigg ) \leq C_{1,T}\\
&+C_{2,T}\int_{0}^{t} \frac{1}{\sqrt{t-\theta}}\sup_{r \in [0,1]} \mathbf{E} \bigg (\int_{0}^{\theta} \int_{0}^{1}|D_{s,y}u_{n,k}(\theta,r)|^{2}  dyds \bigg )d\theta.
\end{split}
\end{equation}
So setting $B_{n,k}(t)=\sup_{x \in [0,1]}\mathbf{E} \bigg (\int_{0}^{t} \int_{0}^{1} |D_{s,y}u_{n,k+1}(t,x)|^2 dyds \bigg )$ and iterating this inequality we have

\begin{equation}  \label{iteration}
\begin{split}
& B_{n,k+1}(t) \leq C_{1,T}+C_{2,T}\int_{0}^{t} \frac{1}{\sqrt{t-\theta}} B_{n,k}(\theta) d\theta \\
& \leq C_{T}+C_{T}\int_{0}^{t}\int_{0}^{\theta}\frac{1}{\sqrt{t-
\theta}\sqrt{\theta-s}} B_{n,k-1}(s) ds d\theta \\
& \leq C_{T}+C_{T}\int_{0}^{t}B_{n,k-1}(s) \int_{s}^{t}\frac{1}{\sqrt{t-
\theta}\sqrt{\theta-s}} d\theta ds \\
&\leq C_{T}+C_{T}\int_{0}^{t}B_{n,k-1}(s)ds,
\end{split}
\end{equation}
\end{remark}
where we used that $\int_{s}^{t}\frac{1}{\sqrt{t-
\theta}\sqrt{\theta-s}} d\theta < \infty$. Iterating this inequality gives us a uniform bound on $k$ that only depends on the final time $T$, and we get 
$$\sup_{k}\sup_{x \in [0,1], t\in [0,T]}\mathbf{E} \bigg (\int_{0}^{t} \int_{0}^{1} |D_{s,y}u_{n,k+1}(t,x)|^2 dyds \bigg ) \leq C_{T}.$$

\smallskip
{\it{Step 3: The process $u_n$ belongs to the space $D_{1,2}$.}} \\

\noindent We have shown that $u_{n,k}$ converges to $u_n$ in the $P_{2,\infty}$ norm and $\sup_{t\in [0,T], x\in [0,1]}\mathbf{E}(u_{n,k}(t,x)^{2}) \leq C_{T}$, so we get that $\sup_{t\in [0,T], x\in [0,1]}\mathbf{E}(u_{n}(t,x)^{2}) \leq C_{T}.$
\\
Moreover, in Remark $\ref{iter}$ of Step 2 we showed that for all $t \in [0,T], x\in [0,1]$
\begin{equation} \label{requirem1}
\sup_{  k}\mathbf{E} \Big (\|D_{\cdot,\cdot}u_{n,k}  (t,x)\|^2_{L^{2}(  [0,T] \times [0,1])} \Big )\leq C_{T}.
\end{equation}
Hence, the conditions of Lemma 1.2.3 of \cite{N} are satisfied and we get that $u_{n} \in D_{1,2}$ and $D_{s,y}u_{n,k}(t,x) \rightarrow D_{s,y}u_{n}(t,x)$ in the weak topology. In addition, the process $u_n$ is an $\{\mathcal{F}_s, \mbox{ } s\le t\}$-adapted process, generated by the Wiener process as in \cite{W6}. We can now differentiate $(\ref{un})$ since $u_{n} \in D_{1,2}$ and get
\begin{equation} \label{PPddfo3}
\begin{split}
 D_{s,y}&u_{n}(t,x) =G_{t-s}(x,y)H_n(|u_{n}(s,y)|)\sigma(s,y,u_{n}(s,y)) \\
&+\int_{s}^{t}\int_{0}^{1}G_{t-\theta}(x,r)m(n)(\theta,r)D_{s,y}u_{n}(\theta,r) dr d \theta  \\
&+ \int_{s}^{t}\int_{0}^{1} G_{t-\theta}(x,r) \hat{m}(n)(\theta,r)D_{s,y}u_{n}(\theta,r) W(dr,d \theta),
\end{split} 
\end{equation}
where $m$ and $\hat{m}$ are bounded. Note that this holds for $s \le t$, and if $s>t$, then $D_{s,y}u_{n}(t,x)=0$. \\

%{In fact the solution to the integral %equation $(\ref{PPddfo3})$ is unique.
%we suppose that $ \hat{D}_{s,y}u_{n}(t,x)$ %is another solution of $(\ref{PPddfo3})$, %then
%\begin{equation} \nonumber
%\begin{split}
 %|D_{s,y}&u_{n}(t,x)-\hat{D}_{s,y}u_{n}%(t,x)|\leq  \\
% &\leq \int_{s}^{t}\int_{0}^{1}|G_{t-\theta}(x,r)|m_2(n)(r,\theta)|D_{s,y}u_{n}(\theta,r)-\hat{D}_{s,y}u_{n}(\theta,r)|drd \theta\\
 %& +\Big |\int_{s}^{t}\int_{0}^{1} G_{t-\theta}(x,r) \hat{m}_2(n)(r,\theta)(D_{s,y}u_{n}(\theta,r)- \hat{D}_{s,y}u_{n}(\theta,r))W(dr,d \theta) \Big |.
%\end{split} 
%\end{equation}
%Then, following the same argument as in equation $(\ref{OoOf})$ we get that ,
%\begin{equation} \nonumber
%\begin{split}
%B_n(t)&\le  c_{t,n}\int_{0}^{t}B_n(s)d s+c_n\int_{0}^{t}B_n(\theta)d \theta \le c_{t,n} \int_{0}^{t}B_n(\theta)d \theta
%\end{split} 
%\end{equation}
%for $\theta>s$ and for
% $$B_n(t):=\mathbf{E}\Big \{\int_{0}^{t}\int_{0}^{1}\|D_{s,y}u_{n}(t,\cdot) -\hat{D}_{y,s}u_{n}(t,\cdot)\|^2 dyds \Big \}.$$  
% According to the Gronwall's inequality,
%$$B_n(t):=\mathbf{E}\Big \{\int_{0}^{t}\int_{0}^{1}\|D_{s,y}u_{n}(t,\cdot) -\hat{D}_{y,s}u_{n}(t,\cdot)\|^2 dyds \Big \} = 0$$
%and we get that, $D_{s,y}u_{n}(t,x) =\hat{D}_{s,y}u_{n}(t,x)$ for $x \in [0,1]$ and $t \in [0,T]$. Hence, the solution to the equation $(\ref{PPddfo3})$ is unique. \\
%

{\it{Step 4: The process $u_n$ belongs to the space $L_{1,2}$.}}\\

According to the definition of the space $L_{1,2}$, i.e., relation $(\ref{SJ29d})$, it suffices to prove that  
$$\mathbf{E}\Big(\int_0^T\int_{0}^{1}\int_0^T\int_{0}^{1}|D_{s,y}u_n(t,x)|^2dydsdxdt\Big{)}<\infty.$$
For  a given $(t,x)$ we have that
\begin{equation} \label{ddfogn974h}
\begin{split}
&\mathbf{E}\Big(\int_0^T\int_{0}^{1}\int_0^T\int_{0}^{1}|D_{s,y}u_n(t,x)|^2dydsdxdt\Big{)}
\le \int_0^T\int_{0}^{1}\mathbf{E}\Big(\int_0^T\int_{0}^{1}|D_{s,y}u_n(t,x)|^2dyds\Big{)}dxdt\\
&=\int_0^T\int_{0}^{1}\mathbf{E}\Big(\int_0^T\int_{0}^{1}|D_{s,y}u_n(t,x)-D_{s,y}u_{n,k}(t,x)+D_{s,y}u_{n,k}(t,x)|^2dyds\Big{)}dxdt \\
&\leq
c\int_0^T\int_{0}^{1}\mathbf{E}\Big(\int_0^T\int_{0}^{1}|D_{s,y}u_n(t,x)-D_{s,y}u_{n,k}(t,x)|^2
dyds\Big{)}dxdt\\
&+c\int_0^T\int_{0}^{1}\mathbf{E}\Big(\int_0^T\int_{0}^{1} |D_{s,y}u_{n,k}(t,x)|^2dyds\Big{)}dxdt.
\end{split}
\end{equation}
The last term in the above inequality is bounded since
\begin{equation} \label{nonst4}
\int_0^T\int_{0}^{1}\mathbf{E}\Big(\int_0^T\int_{0}^{1} |D_{s,y}u_{n,k}(t,x)|^2dyds\Big{)}dxdt \leq \int_0^T\int_{0}^{1}\sup_{t \in [0,T], x \in [0,1]} \mathbf{E}\Big(\int_0^t\int_{0}^{1} |D_{s,y}u_{n,k}(t,x)|^2dyds\Big{)}dxdt \leq C_{T}.
\end{equation}
Now, we only need to prove the following
\begin{equation}\nonumber
\begin{split}
\sup_{x \in [0,1],t \in [0,T]}\mathbf{E}\Big(\int_0^T\int_{0}^{1}|D_{s,y}&u_n(t,x)-D_{s,y}u_{n,k}(t,x)|^2
dyds\Big{)} <\infty.
\end{split}
\end{equation}

Subtracting $(\ref{OoOf})$ from $(\ref{PPddfo3})$ gives us
\begin{equation} \nonumber
\begin{split}
 |D_{s,y}&u_{n}(t,x)- D_{s,y}u_{n,k+1}(t,x)| \\
 &\le|G_{t-s}(x,y)| |H_n(|u_{n}(s,y)|)\sigma(s,y,u_{n}(s,y))-H_n(|u_{n,k}(s,y)|)\sigma(s,y,u_{n,k}(s,y)) | \\
&+\int_{s}^{t}\int_{0}^{1}|G_{t-\theta}(x,r)|  |m(n)(\theta,r)D_{s,y}u_{n}(\theta,r)-m_2(n,k)(\theta,r)D_{s,y}u_{n,k}(\theta,r) |drd \theta  \\
&+\Big | \int_{s}^{t}\int_{0}^{1} G_{t-\theta}(x,r) \Big (\hat{m}(n)(\theta,r)D_{s,y}u_{n}(\theta,r)-\hat{m}_2(n,k)(\theta,r)D_{s,y}u_{n,k}(\theta,r) \Big)W(dr,d \theta)\Big |.
\end{split} 
\end{equation}
The previous inequality holds on for $s \le t$, otherwise, i.e., if $s>t$ then $$ D_{s,y}u_{n}(t,x)- D_{s,y}u_{n,k+1}(t,x)=0.$$
But since
\begin{equation}  \label{Ox1}
\begin{split}
|m(n)(\theta,r)&D_{s,y}u_{n}(\theta,r) -m_2(n,k)(\theta,r)D_{s,y}u_{n,k}(\theta,r)| \\
&\le |m(n)(\theta,r)| |D_{s,y}u_{n}(\theta,r)-D_{s,y}u_{n,k}(\theta,r)| \\
&+|D_{s,y}u_{n,k}(\theta,r)||m(n)(\theta,r)-m_2(n,k)(\theta,r)|,
\end{split} 
\end{equation}
we get
\begin{equation} \nonumber
\begin{split}
 |D_{s,y}&u_{n}(t,x)- D_{s,y}u_{n,k+1}(t,x)| \\
 &\le|G_{t-s}(x,y)| |H_n(|u_{n}(s,y)|)\sigma(s,y,u_{n}(s,y))-H_n(|u_{n,k}(s,y)|)\sigma(s,y,u_{n,k}(s,y)) | \\
&+\int_{s}^{t}\int_{0}^{1}|G_{t-\theta}(x,r)|   |m(n)(\theta,r)| |D_{s,y}u_{n}(\theta,r)-D_{s,y}u_{n,k}(\theta,r)|drd \theta  \\
&+\int_{s}^{t}\int_{0}^{1}|G_{t-\theta}(x,r)|  |D_{s,y}u_{n,k}(\theta,r)||m(n)(\theta,r)-m_2(n,k)(\theta,r)|drd \theta  \\
&+\Big | \int_{s}^{t}\int_{0}^{1} G_{t-\theta}(x,r) \Big (\hat{m}(n)(\theta,r)D_{s,y}u_{n}(\theta,r)-\hat{m}_2(n,k)(\theta,r)D_{s,y}u_{n,k}(\theta,r) \Big )W(dr,d \theta)\Big |.
\end{split} 
\end{equation}
Equivalently,
\begin{equation} \nonumber
\begin{split}
 |D_{s,y}&u_{n}(t,x)- D_{s,y}u_{n,k+1}(t,x)| \\
 &\le|G_{t-s}(x,y)| |H_n(|u_{n}(s,y)|)\sigma(s,y,u_{n}(s,y))-H_n(|u_{n,k}(s,y)|)\sigma(s,y,u_{n,k}(s,y)) | \\
&+\int_{s}^{t}\int_{0}^{1}|G_{t-\theta}(x,r)|   |m(n)(\theta,r)| |D_{s,y}u_{n}(\theta,r)-D_{s,y}u_{n,k}(\theta,r)|drd \theta  \\
&+\int_{s}^{t}\int_{0}^{1}|G_{t-\theta}(x,r)|  |D_{s,y}u_{n,k}(\theta,r)||m(n)(\theta,r)|drd \theta  \\
&+\int_{s}^{t}\int_{0}^{1}|G_{t-\theta}(x,r)|  |D_{s,y}u_{n,k}(\theta,r)||m_2(n,k)(\theta,r)|drd \theta  \\
&+\Big | \int_{s}^{t}\int_{0}^{1} G_{t-\theta}(x,r) \Big (\hat{m}(n)(\theta,r)D_{s,y}u_{n}(\theta,r)-\hat{m}_2(n,k)(\theta,r)D_{s,y}u_{n,k}(\theta,r) \Big )W(dr,d \theta)\Big |.
\end{split} 
\end{equation}
Using the fact that $|m(n)(\theta,r)| \le C_{f}, |m_2(n,k)(\theta,r)| \le C_{\sigma}$ for constants independent of $k$, we have
\begin{equation} \nonumber
\begin{split}
 |D_{s,y}&u_{n}(t,x)- D_{s,y}u_{n,k+1}(t,x)| \\
 &\le|G_{t-s}(x,y)| |H_n(|u_{n}(s,y)|)\sigma(s,y,u_{n}(s,y))-H_n(|u_{n,k}(s,y)|)\sigma(s,y,u_{n,k}(s,y)) | \\
&+c(n)\int_{s}^{t}\int_{0}^{1}|G_{t-\theta}(x,r)|    |D_{s,y}u_{n}(\theta,r)-D_{s,y}u_{n,k}(\theta,r)|drd \theta  \\
&+c(n)\int_{s}^{t}\int_{0}^{1}|G_{t-\theta}(x,r)|  |D_{s,y}u_{n,k}(\theta,r)|drd \theta  \\
&+\Big | \int_{s}^{t}\int_{0}^{1} G_{t-\theta}(x,r) \Big (\hat{m}(n)(\theta,r)D_{s,y}u_{n}(\theta,r)-\hat{m}_2(n,k)(\theta,r)D_{s,y}u_{n,k}(\theta,r) \Big )W(dr,d \theta)\Big |.
\end{split} 
\end{equation}
As a result we get 
\begin{equation} \nonumber
\begin{split}
 |D_{s,y}&u_{n}(t,x)- D_{s,y}u_{n,k+1}(t,x)|^{2} \\
 &\le c(|G_{t-s}(x,y)| |H_n(|u_{n}(s,y)|)\sigma(s,y,u_{n}(s,y))-H_n(|u_{n,k}(s,y)|)\sigma(s,y,u_{n,k}(s,y)) |)^2 \\
&+c(n) \Big (\int_{s}^{t}\int_{0}^{1}|G_{t-\theta}(x,r)|    |D_{s,y}u_{n}(\theta,r)-D_{s,y}u_{n,k}(\theta,r)|drd \theta \Big )^2  \\
&+c(n) \Big (\int_{s}^{t}\int_{0}^{1}|G_{t-\theta}(x,r)|  |D_{s,y}u_{n,k}(\theta,r)|drd \theta \Big )^2  \\
&+c \Big | \int_{s}^{t}\int_{0}^{1} G_{t-\theta}(x,r) \Big (\hat{m}(n)(\theta,r)D_{s,y}u_{n}(\theta,r)-\hat{m}_2(n,k)(\theta,r)D_{s,y}u_{n,k}(\theta,r) \Big )W(dr,d \theta)\Big |^2.
\end{split} 
\end{equation}
Integrating for $y \in [0,1], s\in [0,t]$ gives us
\begin{equation} \nonumber
K(t,x ;k):=\mathbf{E}\Big ( \int_{0}^{t}\int_{0}^{1}  |D_{s,y}u_{n}(t,x)- D_{s,y}u_{n,k+1}(t,x)|^2 dyds \Big) \le C(n) \sum_{j=1}^{4} K_j(t,x;k),
\end{equation}
where
$$K_1(t,x;k):=\mathbf{E}\Big (\int_{0}^{t}\int_{0}^{1}\Big(|G_{t-s}(x,y)|^2 |H_n(|u_{n}(s,y)|)\sigma(s,y,u_{n}(s,y))-H_n(|u_{n,k}(s,y)|)\sigma(s,y,u_{n,k}(s,y)) |^2\Big) dyds \Big ),$$
$$K_2(t,x;k):=\mathbf{E}\Big (\int_{0}^{t}\int_{0}^{1} \Big (\int_{s}^{t}\int_{0}^{1}|G_{t-\theta}(x,r)|    |D_{s,y}u_{n}(\theta,r)-D_{s,y}u_{n,k}(\theta,r)|drd \theta \Big )^2dyds \Big ) ,$$
$$K_3(t,x;k):=\mathbf{E}\Big (\int_{0}^{t}\int_{0}^{1}\Big (\int_{s}^{t}\int_{0}^{1}|G_{t-\theta}(x,r)|  |D_{s,y}u_{n,k}(\theta,r)|drd \theta \Big )^2dyds \Big ),$$
\begin{equation}
\begin{split}
K_4(t,x;k)&:=\mathbf{E}\Big (\int_{0}^{t}\int_{0}^{1} \Big | \int_{s}^{t}\int_{0}^{1} G_{t-\theta}(x,r) \Big (\hat{m}(n)(\theta,r)D_{s,y}u_{n}(\theta,r) \\
&-\hat{m}_2(n,k)(\theta,r)D_{s,y}u_{n,k}(\theta,r) \Big )W(dr,d \theta)\Big |^2dyds \Big ).
\end{split}
\end{equation}
The term $K_1(t,x;k)$ can be bounded as follows:
\begin{equation}
\begin{split}
&K_1(t,x;k):=\mathbf{E}\Big (\int_{0}^{t}\int_{0}^{1}\Big(|G_{t-s}(x,y)|^2 |H_n(|u_{n}(s,y)|)\sigma(s,y,u_{n}(s,y))-H_n(|u_{n,k}(s,y)|)\sigma(s,y,u_{n,k}(s,y)) |^2\Big) dy ds \Big ) \\ 
& \leq c\mathbf{E}\Big (\int_{0}^{t}\int_{0}^{1}\Big(|G_{t-s}(x,y)|^2 (|H_n(|u_{n}(s,y)|)\sigma(s,y,u_{n}(s,y))|^{2}+|H_n(|u_{n,k}(s,y)|)\sigma(s,y,u_{n,k}(s,y)) |^2)\Big) dy ds \Big ) \\
& \leq c\mathbf{E}\Big (\int_{0}^{t}\int_{0}^{1} |G_{t-s}(x,y)|^2 dy ds \Big )
\leq c\int_{0}^{t}\int_{0}^{1} |G_{t-s}(x,y)|^2  dy ds \leq C_{T},
\end{split}
\end{equation}
As a result we get that 
$$
\sup_{t\in [0,T], x \in [0,1]} K_1(t,x;k) \leq C_{T}.$$

\noindent We can bound the term $K_2(t,x;k)$ as follows
\begin{equation} \nonumber
\begin{split}
&K_2(t,x;k):=\mathbf{E}\Big (\int_{0}^{t}\int_{0}^{1} \Big (\int_{s}^{t}\int_{0}^{1}|G_{t-\theta}(x,r)|    |D_{s,y}u_{n}(\theta,r)-D_{s,y}u_{n,k}(\theta,r)|drd \theta \Big )^2dyds \Big )  \\
&\leq   \mathbf{E}\Big (\int_{0}^{t}\int_{0}^{1} 
\Big(\int_{s}^{t}\int_{0}^{1}|G_{t-\theta}(x,r)|^{2} dr d\theta  \int_{s}^{t}\int_{0}^{1} |D_{s,y}u_{n}(\theta,r)-D_{s,y}u_{n,k}(\theta,r)|^{2} drd \theta \Big) dyds \Big )               \\
& \leq  C_{T}  \mathbf{E}\Big (\int_{0}^{t}\int_{0}^{1}  \int_{s}^{t}\int_{0}^{1} |D_{s,y}u_{n}(\theta,r)-D_{s,y}u_{n,k}(\theta,r)|^{2} drd \theta dyds \Big )  \\
&= C_{T} \int_{0}^{t}\int_{0}^{1}  \mathbf{E}\Big (\int_{0}^{\theta}\int_{0}^{1} |D_{s,y}u_{n}(\theta,r)-D_{s,y}u_{n,k}(\theta,r)|^{2} dyds \Big ) dr d\theta \\
&\leq C_{T} \int_{0}^{t}\sup_{r\in [0,1]}  \mathbf{E}\Big (\int_{0}^{\theta}\int_{0}^{1} |D_{s,y}u_{n}(\theta,r)-D_{s,y}u_{n,k}(\theta,r)|^{2} dyds \Big ) d\theta. \\
\end{split}
\end{equation}
\normalsize

\noindent The term $K_3$ can also be bounded in a similar way as $B_1(t)$ and we get directly that
\begin{equation} \nonumber
\begin{split}
&K_3(t,x;k):= \mathbf{E}\Big (\int_{0}^{T}\int_{0}^{1}\Big (\int_{s}^{t}\int_{0}^{1}|G_{t-\theta}(x,r)|  |D_{s,y}u_{n,k}(\theta,r)|drd \theta \Big )^2dyds \Big )\\
& \leq c \int_{0}^{t}  \mathbf{E} \bigg (\int_{0}^{\theta} \int_{0}^{1} \|D_{s,y}u_{n,k}(\theta,\cdot)\|_{L^2([0,1])}^{2} dy ds \bigg )  d \theta.
\end{split}
\end{equation}
So according to $(\ref{UD28})$, we have that
$$\sup_{t \in [0,T], x\in [0,1]} K_3(t,x;k)\leq C_{T} .$$

\noindent Now we deal with the stochastic term $K_4(t,x;k).$ 
\begin{equation} \nonumber
\begin{split}
K_4(t,x;k)\le c\int_{0}^{t}\int_{0}^{1}\mathbf{E} \bigg ( & \Big | \int_{s}^{t}\int_{0}^{1}G_{t-\theta}(x,r)[\hat{m}(n)(\theta,r)D_{s,y}u_{n}(\theta,r) \\
&-\hat{m}_2(n,k)(\theta,r)D_{s,y}u_{n,k}(\theta,r)] W(dr,d \theta) \Big |^2\bigg )dyds.
\end{split}
\end{equation}
The Burkholder-Davis-Gundy inequality yields,
\begin{equation} \nonumber
\begin{split}
K_4(t,x;k)&\le c\int_{0}^{t}\int_{0}^{1}\mathbf{E} \bigg (  \int_{s}^{t}\int_{0}^{1}|G_{t-\theta}(x,r)|^2|\hat{m}(n)(\theta,r)D_{s,y}u_{n}(\theta,r)\\
&-\hat{m}_2(n,k)(\theta,r)D_{s,y}u_{n,k}(\theta,r)|^2drd \theta  \bigg )dyds.
\end{split} 
\end{equation}
But as in $(\ref{Ox1})$ we can get 
\begin{equation} \nonumber
\begin{split}
&K_4(t,x;k) \leq c \int_{0}^{t} \int_{0}^{1} \mathbf{E} \bigg( 
      \int_{s}^{t} \int_{0}^{1} 
    \bigg\{ E_1(t,x;k)^2 + E_2(t,x;k)^2
    \bigg\} \, dr \, d\theta 
    \bigg) \, dy \, ds
\end{split}
\end{equation}
for
$$E_1(t,x;k):=|G_{t-\theta}(x,r)||\hat{m}(n)(\theta,r)| |D_{s,y}u_{n}(\theta,r)-D_{s,y}u_{n,k}(\theta,r)|$$
$$E_2(t,x;k):=|G_{t-\theta}(x,r)||D_{s,y}u_{n,k}(\theta,r)||\hat{m}(n)(\theta,r)-\hat{m}_2(n,k)(\theta,r)|.$$

Thus,
\begin{equation} \label{qq55go0f}
\begin{split}
K_4(t,x;k)&\le
 c\int_{0}^{t}\int_{0}^{1}\mathbf{E} \Big \{ \int_{s}^{t}\int_{0}^{1} E_1(t,x;k)^2 drd \theta  \Big \}dyds \\
&+ c\int_{0}^{t}\int_{0}^{1}\mathbf{E} \Big \{ \int_{s}^{t}\int_{0}^{1}E_2(t,x;k)^2 drd \theta \Big \}dyds. \\
\end{split} 
\end{equation}
Let's start with the first term of $K_4(t,x;k)$, replacing the quantity of $E_1(t,x;k)$. Then,
\begin{equation} \nonumber
\begin{split}
c\int_{0}^{t}\int_{0}^{1}&\mathbf{E} \Big \{  \int_{s}^{t}\int_{0}^{1} E_1(t,x;k)^2 drd \theta  \Big \}dyds \\
&=c\int_{0}^{t}\int_{0}^{1}\mathbf{E} \Big \{  \int_{s}^{t} \int_{0}^{1} |G_{t-\theta}(x,r)|^2|\hat{m}(n)(\theta,r)|^2 |D_{s,y}u_{n}(\theta,r)-D_{s,y}u_{n,k}(\theta,r)|^2 drd \theta  \Big \}dyds.
\end{split} 
\end{equation}
Since $\hat{m}(n)$ is bounded by a constant $C_{\sigma}$ we get, 
\begin{equation} \nonumber
\begin{split}
c\int_{0}^{t}\int_{0}^{1}&\mathbf{E} \Big \{  \int_{s}^{t}\int_{0}^{1} E_1(t,x;k)^2 drd \theta \Big \}dyds \\
&\le C_{\sigma}\int_{0}^{t}\int_{0}^{1}\mathbf{E} \Big \{ \int_{s}^{t}\int_{0}^{1} |G_{t-\theta}(x,r)|^2 |D_{s,y}u_{n}(\theta,r)-D_{s,y}u_{n,k}(\theta,r)|^2 drd \theta  \Big \}dyds.
\end{split} 
\end{equation}
The right-hand side of the previous inequality has the same form as the term $\Gamma_2(t)$, so we get in the same way

\begin{equation} \label{jov9flg1}
\begin{split}
c\int_{0}^{t}\int_{0}^{1}&\mathbf{E} \Big \{  \int_{s}^{t}\int_{0}^{1} E_1(t,x;k)^2 drd \theta  \Big \}dyds \\
&\le c\int_{0}^{t} \frac{1}{\sqrt{t-\theta}} \sup_{r \in [0,1]} \mathbf{E} \bigg(\int_{0}^{\theta} \int_{0}^{1}  |D_{s,y}u_{n}(\theta,r)-D_{s,y}u_{n,k}(\theta,r)|^{2} dy ds \bigg)  d\theta .
\end{split}
\end{equation} 
For the second term of $K_4(t,x;k)$, we have
\begin{equation} \nonumber
\begin{split}
 c\int_{0}^{t}\int_{0}^{1}&\mathbf{E} \Big \{\int_{s}^{t}\int_{0}^{1}E_2(t,x;k)^2 dr d\theta  \Big \}dyds \\
&\le c\int_{0}^{t}\int_{0}^{1}\mathbf{E} \Big \{\int_{s}^{t}\int_{0}^{1}|G_{t-\theta}(x,r)|^2|D_{s,y}u_{n,k}(\theta,r)|^2 \Big [|\hat{m}(n)(\theta,r)|\\
&+|\hat{m}_2(n,k)(\theta,r)|\Big ]^2 dr 
 d\theta  \Big \}dyds,
\end{split} 
\end{equation}
where we used the triangle inequality.
But since $\hat{m}, \hat{m}_{2}$ are both bounded we have that 
\begin{equation} \nonumber
\begin{split}
 c\int_{0}^{t}\int_{0}^{1}\mathbf{E} &\Big \{\int_{s}^{t}\int_{0}^{1}E_2(t,x;k)^2 drd \theta  \Big \}dyds \\
&\le c\int_{0}^{t}\int_{0}^{1}\mathbf{E} \Big \{\int_{s}^{t}\int_{0}^{1}|G_{t-\theta}(x,r)|^2|D_{s,y}u_{n,k}(\theta,r)|^2 drd \theta \Big \}dyds.
\end{split} 
\end{equation}
We have bounded the right-hand side of the above inequality in $\ref{gamma_2bound}$, and eventually found a bound that does not depend on $k$, so we can get  
\begin{equation} \label{jov9flg2}
 c\int_{0}^{t}\int_{0}^{1}\mathbf{E} \Big \{\int_{s}^{t}\int_{0}^{1}E_2(t,x;k)^2 drd \theta  \Big \}dyds  \le C_{T}.
\end{equation}

Combining all the above for $K_4(t,x;k)$ we get,
\begin{equation} \nonumber
\begin{split}
K_4(t,x;k) \le C_{T}+ c \int_{0}^{t} \frac{1}{\sqrt{t-\theta}} \sup_{r \in [0,1]} \mathbf{E} \bigg( \int_{0}^{\theta} \int_{0}^{1}  |D_{s,y}u_{n}(\theta,r)-D_{s,y}u_{n,k}(\theta,r)|^{2} dy ds \bigg) d\theta. \\
\end{split} 
\end{equation}

\noindent We now combine all of these estimates for $(K_i(t,x;k))_{i=1}^{4}$. \\
We take $F_{n,k}(t)=\sup_{x \in [0,1]} \mathbf{E} \bigg(\int_{0}^{t} \int_{0}^{1}  |D_{s,y}u_{n}(t,x)-D_{s,y}u_{n,k}(t,x)|^{2} dy ds \bigg)$. Then we have that 

\begin{equation} \nonumber
\begin{split}
F_{n,k+1}(t) \le C_{T}+C_{T}\int_{0}^{t} F_{n,k}(\theta) d\theta+ C_{T}\int_{0}^{t} \frac{1}{\sqrt{t-\theta}} F_{n,k}(\theta) d\theta. \\
\end{split} 
\end{equation}

\noindent We notice that the previous inequality is similar to $(\ref{Ox12_2})$. So, following the same steps as in Remark $\ref{iter}$ of Step 2, we get $\sup_{t \in [0,T]}F_{n,k}(t) \leq C_{T}$ and thus

\begin{equation} \label{ddf8mnmn}
\begin{split}
\sup_{ t \in [0,T], x\in [0,1]}\mathbf{E}\Big(\int_0^T\int_{0}^{1}|D_{s,y}&u_n(t,x)-D_{s,y}u_{n,k}(t,x)|^2
dyds\Big{)} <\infty.
\end{split}
\end{equation}
Then directly by $(\ref{ddfogn974h})$ and $(\ref{ddf8mnmn})$ we get
$$\mathbf{E}\Big(\int_0^T\int_{0}^{1}\int_0^T\int_{0}^{1}|D_{s,y}u_n(t,x)|^2dydsdxdt\Big{)}< \infty,$$
which yields that $u_n$ belongs to the space $L_{1,2}$. \\
%%%%%%%%%%%%%%%%%%%%%%%%%%%%%%%%%%%%%%%%%%%%
%%%%%%%%%%%%%%%%

\noindent We are now ready to prove the first part of Theorem $\ref{theorem21}$, namely that $u \in L_{1,2}^{loc}$ for the mild solution $u$ to the stochastic heat equation $(\ref{smr00})$.
According to the definition of $L^{loc}_{1,2}$, it suffices to show that there exists a pair of sequences $ (\Omega_n, u_n )$ with $\{ (\Omega_n, u_n ); \Omega_{n} \subseteq \Omega \}_{n \ge 1} \subset \mathcal{F} \times L_{1,2}$ such that
\begin{enumerate}
\item $\Omega_n \uparrow \Omega$, a.s
\item $u=u_n$ a.s. on $\Omega_n, $
\end{enumerate}
where  $\Omega_{n}$ is given by $(\ref{Omega4})$.
It is well-known (for instance from Theorem 2.4.3 in \cite{N}) that $(\ref{smr00})$ has a unique adapted solution $\{u(t, x) ; t\in [0,T], x\in [0,1]\}$ with continuous paths. Then, according to the definition of $\Omega_n$, we get that $\Omega_n \uparrow \Omega$, and we have the convergence
$$\lim_{n\rightarrow \infty}P(\Omega_n)=P(\Omega)=1.$$
On the other hand,  Lemma $(\ref{Ox5})$ establishes the existence and uniqueness of the solution $u_n$ of equation $(\ref{mild0})$, for $ (t,x) \in [0,1] \times [0,T]$, and step 4 of this proof ensures that $u_n$ belongs to the space $L_{1,2}$. Consequently, $u_n$ constructs a well-defined local approximation of $u$ in the space $\Omega_n$, in the sense of $L^{loc}_{1,2}$ with $u_n=u$ on $\Omega_n$. We conclude that $u \in L^{loc}_{1,2}$. \\

%%%%%%%%%%%%%%%%%%%%
%%%%%%%%%%%%%%%%

%%%%%%%%%%%%%%%%%%%%%%%%%%%%%%%%%%%%%%%%%%%%

\section{Existence of density}

%LEMMA 1DENSITY

\begin{lemma} \label{UD82}
Let $u_n$ be the solution of the equation $(\ref{mild0})$, with coefficients $f,\sigma$ that satisfy the assumptions (1) and (2) of Theorem {\ref{theorem21}}. For any $\hat{s}>0$, $\varepsilon<\min\{1,\hat{s}\}$ and for a positive constant $C$ that depends only on $n$, i.e., $ C:=C(n)>0$, we get the following estimates for the Malliavin derivative
\begin{equation}\label{d84un}
\sup_{t\in[\hat{s}-\varepsilon,\hat{s}]} \sup_{x \in [0,1]} \mathbf{E}\Big{(}\int_{\hat{s}-\varepsilon}^{\hat{s}} \int_{0}^{1}
|D_{s,y}u_{n}(t,x)|^{2} dy ds\Big{)}< C \varepsilon^{\frac{1}{2}}.
\end{equation}
Moreover, setting  
$$ 
b(x,y,s,t)=D_{s,y}u_{n}(t,x)-G_{t-s}(x,y)H_n(|u_{n}(s,y)|)\sigma(s,y,u_{n}(s,y))| 
 $$
we have that 
\begin{equation} \label{b2bound}
\begin{split}
&E\Big(\int_{t-\varepsilon}^{t}\int_{0}^{1} b^2(x,y,s,t)dyds\Big) \leq c \max\{  \varepsilon^{2}, \varepsilon^{\frac{1}{2}} \int_{t-\varepsilon}^{t} \sum_{k=1}^{\infty}  e^{-2k^2 \pi^2 (t-\theta)} \sin^{2}( k \pi x) d\theta \}.
\end{split} 
\end{equation}
\end{lemma}

\begin{proof}
    
We start with equation $(\ref{PPddfo3})$,
\begin{equation} \nonumber
\begin{split}
 |D_{s,y}&u_{n}(t,x)|^{2} =c|G_{t-s}(x,y)H_n(|u_{n}(s,y)|)\sigma(s,y,u_{n}(s,y))|^{2} \\
&+c \Big (\int_{s}^{t}\int_{0}^{1}|G_{t-\theta}(x,r)||m(n)(r,\theta)||D_{s,y}u_{n}(\theta,r)| drd \theta  \Big )^{2} \\
&+c \Big  | \int_{s}^{t}\int_{0}^{1} G_{t-\theta}(x,r) \hat{m}(n)(r,\theta)D_{s,y}u_{n}(\theta,r) W(dr,d \theta) \Big |^{2},
\end{split} 
\end{equation}
where $s<t$. We have that
\begin{equation} \nonumber
\begin{split}
\mathbf{E}\Big(\int_{\gamma}^{t}\int_{0}^{1} |D_{s,y}&u_{n}(t,x)|^{2} dy ds \Big) \leq c \int_{\gamma}^{t}\int_{0}^{1} |G_{t-s}(x,y)|^{2} dy ds \\
&+c \mathbf{E}\Big(\int_{\gamma}^{t}\int_{0}^{1} \big(\int_{s}^{t}\int_{0}^{1}|G_{t-\theta}(x,r)||D_{s,y}u_{n}(\theta,r)| drd \theta \big)^{2} dy ds \Big) \\
&+c \mathbf{E}\Big(\int_{\gamma}^{t}\int_{0}^{1} \Big( \int_{s}^{t}\int_{0}^{1} G_{t-\theta}(x,r) D_{s,y}u_{n}(\theta,r) W(dr,d \theta) \Big)^{2} dy ds \Big).
\end{split} 
\end{equation}

\noindent Taking supremum with respect to $x \in [0,1]$ gives us:
\begin{equation} \nonumber
\begin{split}
\sup_{x\in[0,1]}\mathbf{E}\Big(\int_{\gamma}^{t}\int_{0}^{1} |D_{s,y}&u_{n}(t,x)|^{2} dy ds \Big) \leq c \sup_{x \in[0,1]}\int_{\gamma}^{t}\int_{0}^{1} |G_{t-s}(x,y)|^{2} dy ds \\
&+c \sup_{x \in[0,1]}\mathbf{E}\Big(\int_{\gamma}^{t}\int_{0}^{1}\big(\int_{s}^{t}\int_{0}^{1}|G_{t-\theta}(x,r)||D_{s,y}u_{n}(\theta,r)| drd \theta \big)^{2} dy ds \Big) \\
&+c \sup_{x \in[0,1]}\mathbf{E}\Big(\int_{\gamma}^{t}\int_{0}^{1} \Big( \int_{s}^{t}\int_{0}^{1} G_{t-\theta}(x,r) D_{s,y}u_{n}(\theta,r) W(dr,d \theta) \Big)^{2} dy ds \Big)=\\
&:= R_{1}+R_{2}+R_{3}.
\end{split} 
\end{equation}
We set $L_{n}(t)= \sup_{x\in[0,1]}\mathbf{E} \Big( \int_{\gamma}^{t}   \int_{0}^{1} (D_{s,y}u_{n}(t,x))^{2} dy ds \Big)$.
We bound the term $R_{1}$ as follows
\begin{equation} \nonumber
R_{1}=\sup_{x \in [0,1]}\int_{\gamma}^{t} \int_{0}^{1} |G_{t-s}(x,y)|^{2} dyds \leq  \sup_{x \in [0,1]}\int_{\gamma}^{t}  \int_{0}^{1} \frac{e^{-\frac{2(x-y)^{2}}{4(t-s)}}}{t-s} dyds \leq \int_{\gamma}^{t} \frac{1}{\sqrt{t-s}} ds=\sqrt{t-\gamma}.
\end{equation}
\smallskip
Next, we bound the term $R_{2}$
\begin{equation}
\begin{split}
&R_{2}=\sup_{x \in [0,1]}\mathbf{E} \Big (\int_{\gamma}^{t} \int_{0}^{1} \Big(\int_{s}^{t}\int_{0}^{1} |G_{t-\theta}(x,r)||D_{s,y}u_{n}( \theta,r)| dr d\theta \Big)^{2} dyds \Big ) \\
&\leq \sup_{x \in [0,1]}\mathbf{E} \Big (\int_{\gamma}^{t} \int_{0}^{1} \Big(\int_{s}^{t} \frac{1}{\sqrt{t-\theta}}\int_{0}^{1} e^{-\frac{(x-r)^{2}}{4 \pi (t-\theta)}}|D_{s,y}u_{n}(\theta,r)| drd \theta \Big)^{2} dyds \Big ) \\
&\leq  \sup_{x \in [0,1]}\mathbf{E} \Big (\int_{\gamma}^{t}  \int_{0}^{1} \Big(\int_{s}^{t} \frac{1}{\sqrt{t-\theta}} \left(\int_{0}^{1} e^{-\frac{(x-r)^{2}}{2\pi (t-\theta)}} dr \right)^{\frac{1}{2}}\left(\int_{0}^{1} (D_{s,y}u_{n}(\theta,r))^{2} dr \right)^{\frac{1}{2}}d \theta \Big)^{2} dyds \Big ) \\
&\leq c \mathbf{E} \Big (\int_{\gamma}^{t}  \int_{0}^{1} \Big(\int_{s}^{t} \frac{1}{\sqrt{t-\theta}} (t-\theta)^{\frac{1}{4}}\left(\int_{0}^{1} (D_{s,y}u_{n}(\theta,r))^{2} dr \right)^{\frac{1}{2}}d \theta \Big)^{2} dyds \Big ) \\
& \leq c \mathbf{E} \Big (\int_{\gamma}^{t} \int_{0}^{1} \Big(\int_{s}^{t} \frac{1}{(t-\theta)^{\frac{1}{4}}} \left(\int_{0}^{1} (D_{s,y}u_{n}(\theta,r))^{2} dr \right)^{\frac{1}{2}}d \theta \Big)^{2} dyds \Big ) \\
&\leq c \mathbf{E} \Big (\int_{\gamma}^{t} \int_{0}^{1} \int_{s}^{t} \frac{1}{(t-\theta)^{\frac{1}{2}}} d \theta \int_{s}^{t}\int_{0}^{1} (D_{s,y}u_{n}(r, \theta))^{2} dr d \theta  dyds \Big ) \\
& \leq c\mathbf{E} \Big (\int_{\gamma}^{t} \int_{0}^{1} \sqrt{t-s} \int_{s}^{t}\int_{0}^{1} (D_{s,y}u_{n}(r, \theta))^{2} dr d \theta  dyds \Big ) \\
& \leq c \sqrt{T} \mathbf{E} \Big (\int_{\gamma}^{t} \int_{0}^{1} \int_{s}^{t}\int_{0}^{1} (D_{s,y}u_{n}(r, \theta))^{2} dr d \theta  dyds \Big ) \\
& \leq c \sqrt{T}\int_{\gamma}^{t} \int_{0}^{1} \mathbf{E} \Big (  \int_{\gamma}^{\theta} \int_{0}^{1} (D_{s,y}u_{n}(r, \theta))^{2} dy ds   \Big ) dr d\theta \\
& \leq c \sqrt{T}\int_{\gamma}^{t} \sup_{r\in[0,1]}  \mathbf{E} \Big (\int_{\gamma}^{\theta}\int_{0}^{1} \ (D_{s,y}u_{n}(r, \theta))^{2} dy  ds  \Big ) d\theta  \\
& =c \sqrt{T}\int_{\gamma}^{t}   L_{n}(\theta) d\theta.  \\
\end{split}
\end{equation}
\noindent Let us now bound the stochastic term $R_{3}$. We use the Burkholder-Davis-Gundy inequality in the first step.

\begin{equation}
\begin{split}
&R_{3}=\sup_{x \in [0,1]}\mathbf{E} \Big (\int_{\gamma}^{t}  \int_{0}^{1}\Big( \int_{s}^{t}\int_{0}^{1}G_{t-\theta}(x,r) D_{s,y}u_{n}(\theta,r) W(dr,d \theta)\Big)^{2} dyds \Big ) \leq  \\
& \leq c \sup_{x \in [0,1]} \mathbf{E} \Big (\int_{\gamma}^{t} \int_{0}^{1}  \int_{s}^{t}\int_{0}^{1}G^{2}_{t-\theta}(x,r) (D_{s,y}u_{n}(\theta,r))^{2} drd \theta dyds \Big) \\
&  \leq c \sup_{x \in [0,1]} \mathbf{E} \Big (\int_{\gamma}^{t}  \int_{0}^{1}  \int_{s}^{t} \int_{0}^{1} \frac{e^{-\frac{(x-r)^{2}}{2 \pi (t-\theta)}}}{t-\theta}(D_{s,y}u_{n}(\theta,r))^{2} drd \theta dyds \Big) \\
& =c \sup_{x \in [0,1]} \mathbf{E} \Big (\int_{\gamma}^{t}  \int_{0}^{1}  \int_{\gamma}^{\theta} \int_{0}^{1} \frac{e^{-\frac{(x-r)^{2}}{2 \pi (t-\theta)}}}{t-\theta}(D_{s,y}u_{n}(\theta,r))^{2} dyds dr d\theta  \Big) \\
\end{split}
\end{equation}
\begin{equation}
\begin{split}
&  = c \sup_{x \in [0,1]} \int_{\gamma}^{t}  \int_{0}^{1} \frac{e^{-\frac{(x-r)^{2}}{2 \pi (t-\theta)}}}{t-\theta} \mathbf{E} \Big (\int_{\gamma}^{\theta} \int_{0}^{1} (D_{s,y}u_{n}(\theta,r))^{2} dyds \Big) dr d\theta   \\
&  = c \sup_{x \in [0,1]} \int_{\gamma}^{t}  \sup_{r\in [0,1]}\mathbf{E} \Big (\int_{\gamma}^{\theta} \int_{0}^{1} (D_{s,y}u_{n}(\theta,r))^{2} dyds \Big)  \int_{0}^{1} \frac{e^{-\frac{(x-r)^{2}}{2 \pi (t-\theta)}}}{t-\theta} dr d\theta   \\
&  \leq c \int_{\gamma}^{t}  \sup_{r\in [0,1]}\mathbf{E} \Big (\int_{\gamma}^{\theta} \int_{0}^{1} (D_{s,y}u_{n}(\theta,r))^{2} dyds \Big) \frac{1}{\sqrt{t-\theta}} d\theta   \\
&  \leq c\int_{\gamma}^{t} \frac{1}{\sqrt{t-\theta}}  L_{n}(\theta) d\theta.  \\
\end{split}
\end{equation}

\noindent Combining the bounds $R_{1},R_{2},R_{3}$ we obtain:

$$
L_{n}(t) \leq c \int_{\gamma}^{t} L_{n}(\theta) d\theta + c \int_{\gamma}^{t}\frac{1}{\sqrt{t-\theta}} L_{n}(\theta) d\theta + \sqrt{t-\gamma}.
$$

\noindent This is the same kind of equation as $(\ref{df443ffffd})$, so arguing in exactly the same way as we did before, we can get that 

$$
L_{n}(t) \leq \sqrt{t-\gamma}+ C_{T} \int_{\gamma}^{t} L_{n}(\theta) d\theta.
$$

\noindent By Gronwall's inequality we conclude that 

$$
L_{n}(t) \leq C_{T}\sqrt{t-\gamma}.
$$

\noindent As a result, we have that

$$
L_{n}(t)= \sup_{x \in [0,1]}\mathbf{E} \Big( \int_{\gamma}^{t} \int_{0}^{1} (D_{s,y}u_{n}(t,x))^{2} dy ds \Big) \leq C_{T} \sqrt{t-\gamma}.
$$

\noindent Setting $\gamma=\hat{s}-\varepsilon$  in the inequality above and taking into account that $D_{s,y}u_n(t,x)=0$ for $s>t$ we obtain
$$
\sup_{x\in [0,1]}\mathbf{E} \left( \int_{\hat{s}-\varepsilon}^{\hat{s}}  \int_{0}^{1} (D_{s,y}u_{n}(t,x))^{2} dy ds \right) = \sup_{x\in [0,1]}\mathbf{E} \left( \int_{\hat{s}-\varepsilon}^{t} \int_{0}^{1} (D_{s,y}u_{n}(t,x))^{2} dy ds \right) \leq C_{T} \sqrt{t-\hat{s}+\varepsilon}
$$
and taking supremum with respect to $t \in [\hat{s}-\varepsilon, \hat{s}]$ we get 
$$
\sup_{t \in [\hat{s}-\varepsilon, \hat{s}]} \sup_{x\in [0,1]}\mathbf{E} \left( \int_{\hat{s}-\varepsilon}^{\hat{s}}   \int_{0}^{1} (D_{s,y}u_{n}(t,x))^{2} dy ds \right)  \leq C_{T} \varepsilon^{\frac{1}{2}}.
$$

\noindent This proves the desired inequality $(\ref{d84un})$.\\

We are now going to bound 
\begin{equation} \nonumber \label{lldldl0d0dkkd0eer}
E\Big(\int_{t-\varepsilon}^{t}\int_{0}^{1} b^2(x,y,s,t)dyds\Big)=E\Big(\int_{t-\varepsilon}^{t}\int_{0}^{1}|D_{s,y}u_{n}(t,x)-G_{t-s}(x,y)H_n(|u_{n}(s,y)|)\sigma(s,y,u_{n}(s,y))|^2dyds\Big).
\end{equation}
We have that
\begin{equation} \nonumber
\begin{split}
&\Big(D_{s,y}u_{n}(t,x) -G_{t-s}(x,y)H_n(|u_{n}(s,y)|)\sigma(s,y,u_{n}(s,y))\Big)^{2}= \\
&= \Big(\int_{s}^{t}\int_{0}^{1}G_{t-\theta}(x,r)m(n)(\theta,r)D_{s,y}u_{n}(\theta,r) drd \theta + \int_{s}^{t}\int_{0}^{1} G_{t-\theta}(x,r) \hat{m}(n)(\theta,r)D_{s,y}u_{n}(\theta,r) W(dr,d \theta) \Big)^{2}\\
& \leq c \Big(\int_{s}^{t}\int_{0}^{1}G_{t-\theta}(x,r)m(n)(\theta,r)D_{s,y}u_{n}(\theta,r) dr d \theta \Big)^{2}+ c\Big(\int_{s}^{t}\int_{0}^{1} G_{t-\theta}(x,r) \hat{m}(n)(\theta,r)D_{s,y}u_{n}(\theta,r) W(dr,d \theta) \Big)^{2},
\end{split} 
\end{equation}
so, now we have to bound the terms $I_{1},I_{2}$ with
\begin{equation} \nonumber
\begin{split}
&I_{1}= E \Big( \int_{t-\varepsilon}^{t} \int_{0}^{1} \Big(\int_{s}^{t}\int_{0}^{1}G_{t-\theta}(x,r)m(n)(\theta,r)D_{s,y}u_{n}(\theta,r)) drd \theta \Big)^{2} dy d s \Big). \\
&I_{2}= E \Big( \int_{t-\varepsilon}^{t} \int_{0}^{1} \Big(\int_{s}^{t}\int_{0}^{1} G_{t-\theta}(x,r) \hat{m}(n)(\theta,r)D_{s,y}u_{n}(\theta,r) W(dr,d \theta) \Big)^{2} dy ds  \Big). \\
\end{split}
\end{equation}
We start by bounding the term $I_{1}$
\begin{equation} \nonumber
\begin{split}
& I_{1}=E \Big( \int_{t-\varepsilon}^{t} \int_{0}^{1} \Big(\int_{s}^{t}\int_{0}^{1}G_{t-\theta}(x,r)m(n)(\theta,r)(D_{s,y}u_{n}(\theta,r)) dr d \theta \Big)^{2} dy d s \Big) \leq \\ 
& \leq c E \Big( \int_{t-\varepsilon}^{t} \int_{0}^{1} \int_{s}^{t}\int_{0}^{1}G^{2}_{t-\theta}(x,r) dr d \theta \int_{s}^{t} \int_{0}^{1} (D_{s,y}u_n(\theta,r))^{2}dr d\theta dy d s \Big) \\
& \leq c E \Big( \int_{t-\varepsilon}^{t} \int_{0}^{1} \sqrt{t-s} \int_{s}^{t} \int_{0}^{1} (D_{s,y}u_n(\theta,r))^{2}dr d\theta dy d s \Big) \\
& \leq c\sqrt{\varepsilon} E \Big( \int_{t-\varepsilon}^{t} \int_{0}^{1} \int_{s}^{t} \int_{0}^{1} (D_{s,y}u_n(\theta,r))^{2}dr d\theta dy d s \Big) \\
& = c\sqrt{\varepsilon} \int_{t-\varepsilon}^{t} \int_{0}^{1} E \Big( \int_{t-\varepsilon}^{\theta} \int_{0}^{1} (D_{s,y}u_n(\theta,r))^{2}dy ds \Big) dr d\theta \\
& \leq c\varepsilon^{\frac{3}{2}} \sup_{\theta \in [t-\varepsilon,t]} \sup_{r \in [0,1]} E \Big( \int_{t-\varepsilon}^{t} \int_{0}^{1} (D_{s,y}u_n(\theta,r))^{2}dy ds \Big)\\
& \leq c \varepsilon^{2}.
\end{split}
\end{equation}
To bound the term $I_{2}$ we use the Burkholder-Davis-Gundy inequality
\begin{equation} \nonumber
\begin{split}
& I_{2}=E \Big( \int_{t-\varepsilon}^{t} \int_{0}^{1} \Big(\int_{s}^{t}\int_{0}^{1} G_{t-\theta}(x,r) \hat{m}(n)(r,\theta)D_{s,y}u_{n}(\theta,r) W(dr,d \theta) \Big)^{2} dy ds  \Big) \\
& \leq c E \Big( \int_{t-\varepsilon}^{t} \int_{0}^{1} \Big(\int_{s}^{t}\int_{0}^{1} G_{t-\theta}^{2}(x,r) (D_{s,y}u_{n}(\theta,r))^{2} dr d \theta \Big) dy ds  \Big) \\
& =c E \Big( \int_{t-\varepsilon}^{t} \int_{0}^{1} \int_{t-\varepsilon}^{\theta} \int_{0}^{1} G_{t-\theta}^{2}(x,r)   \left(D_{s,y}u_{n}(\theta,r)\right)^{2} dy ds dr d\theta   \Big) \\
\end{split}
\end{equation}
\begin{equation}
\begin{split}
& =c E \Big( \int_{t-\varepsilon}^{t} \int_{0}^{1} G_{t-\theta}^{2}(x,r)  \int_{t-\varepsilon}^{\theta} \int_{0}^{1}  \left(D_{s,y}u_{n}(\theta,r)\right)^{2} dy ds dr d\theta   \Big) \\
& =c  \int_{t-\varepsilon}^{t} \int_{0}^{1} G_{t-\theta}^{2}(x,r)  E \Big(\int_{t-\varepsilon}^{\theta} \int_{0}^{1}  \left(D_{s,y}u_{n}(\theta,r)\right)^{2} dy ds\Big) dr d\theta  \\
& \leq c  \int_{t-\varepsilon}^{t}  \sup_{r \in [0,1]}E \Big(\int_{t-\varepsilon}^{\theta} \int_{0}^{1}  \left(D_{s,y}u_{n}(\theta,r)\right)^{2} dy ds\Big) \int_{0}^{1} G_{t-\theta}^{2}(x,r) dr d\theta  \\
& \leq c \sup_{\theta \in [t-\varepsilon,t]}\sup_{r \in [0,1]}E \Big(\int_{t-\varepsilon}^{t} \int_{0}^{1}  \left(D_{s,y}u_{n}(\theta,r)\right)^{2} dy ds\Big) \int_{t-\varepsilon}^{t}   \int_{0}^{1} G_{t-\theta}^{2}(x,r) dr d\theta  \\
& \leq c \sup_{\theta \in [t-\varepsilon,t]}\sup_{r \in [0,1]}E \Big(\int_{t-\varepsilon}^{t} \int_{0}^{1}  \left(D_{s,y}u_{n}(\theta,r)\right)^{2} dy ds\Big) \int_{t-\varepsilon}^{t}   \int_{0}^{1} \Big( \sum_{k=1}^{\infty}  e^{-k^2 \pi^2 (t-\theta)} \sin( k \pi x) \sin(k \pi r) \Big)^{2} dr  d\theta  \\
& \leq c \sup_{\theta \in [t-\varepsilon,t]}\sup_{r \in [0,1]}E \Big(\int_{t-\varepsilon}^{t} \int_{0}^{1}  \left(D_{s,y}u_{n}(\theta,r)\right)^{2} dy ds\Big) \int_{t-\varepsilon}^{t}    \sum_{k=1}^{\infty}  e^{-2k^2 \pi^2 (t-\theta)} \sin^{2}( k \pi x) d\theta  \\
&\leq c \varepsilon^{\frac{1}{2}} \int_{t-\varepsilon}^{t} \sum_{k=1}^{\infty}  e^{-2k^2 \pi^2 (t-\theta)} \sin^{2}( k \pi x) d\theta. \\
\end{split}
\end{equation}
Combining the two, we conclude that 

\begin{equation} \label{b2bound22} \nonumber
\begin{split}
&E\Big(\int_{t-\varepsilon}^{t}\int_{0}^{1} b^2(x,y,s,t)dyds\Big)=E\Big(\int_{t-\varepsilon}^{t}\int_{0}^{1}|D_{s,y}u_{n}(t,x)-G_{t-s}(x,y)H_n(|u_{n}(s,y)|)\sigma(s,y,u_{n}(s,y))|^2dyds\Big) \\
& \leq c \varepsilon^{2} +  c \varepsilon^{\frac{1}{2}} \int_{t-\varepsilon}^{t} \sum_{k=1}^{\infty}  e^{-2k^2 \pi^2 (t-\theta)} \sin^{2}( k \pi x) d\theta  \leq 2c \max\{  \varepsilon^{2}, \varepsilon^{\frac{1}{2}} \int_{t-\varepsilon}^{t} \sum_{k=1}^{\infty}  e^{-2k^2 \pi^2 (t-\theta)} \sin^{2}( k \pi x) d\theta \}.
\end{split} 
\end{equation}
\end{proof}
\begin{theorem} \label{UD120}
Let $u_n$ be the solution of the equation $(\ref{mild0})$ and assume that the coefficients $f,\sigma$ satisfy the assumptions (1), (2) of Theorem {\ref{theorem21}}. Then for $t \in [0,T]$ and $x \in (0,1)$, 

\begin{equation} \label{dknnv44455ng8jg8jk}
P\Big (\omega \in \Omega_{n} : \int_{0}^{t}\int_{0}^{1}|D_{s,y}u_n(t,x)|^2dyds>0 \Big )=P(\Omega_{n}).
\end{equation}
\end{theorem}

\begin{proof}
 We are thus looking for 
a lower estimate of this equation $(\ref{PPddfo3})$, that is
\begin{equation} \nonumber
\begin{split}
 |D_{s,y}&u_{n}(t,x)|^2 =\Big (G_{t-s}(x,y)H_n(|u_{n}(s,y)|)\sigma(s,y,u_{n}(s,y)) \\
&+\int_{s}^{t}\int_{0}^{1}G_{t-\theta}(x,r)m(n)(r,\theta)D_{s,y}u_{n}(\theta,r) drd \theta  \\
&+ \int_{s}^{t}\int_{0}^{1} G_{t-\theta}(x,r) \hat{m}(n)(r,\theta)D_{s,y}u_{n}(\theta,r) W(dr,d \theta) \Big )^2.
\end{split} 
\end{equation}
Note that the previous inequality holds on for $s \le t$, otherwise, i.e., if $s>t$ then $D_{s,y}u_{n}(t,x)=0$.
If we set 
$$a(x,y,s,t):=G_{t-s}(x,y)H_n(|u_{n}(s,y)|)\sigma(s,y,u_{n}(s,y))$$
and
\begin{equation} \nonumber
\begin{split}
b(x,y,s,t)&:=\int_{s}^{t}\int_{0}^{1}G_{t-\theta}(x,r)m(n)(r,\theta)D_{s,y}u_{n}(\theta,r)) drd \theta  \\
&+ \int_{s}^{t}\int_{0}^{1} G_{t-\theta}(x,r) \hat{m}(n)(r,\theta)D_{s,y}u_{n}(\theta,r)) W(dr,d \theta) \Big .
\end{split} 
\end{equation}
then
\begin{equation} \nonumber
|D_{s,y}u_{n}(t,x)|^2=(a(x,y,s,t)+b(x,y,s,t))^2.
\end{equation}
But, the obvious inequality  $\Big (\frac{1}{\sqrt{2}}a+\sqrt{2}b\Big )^2 \ge 0$  holds for $a,b \in \mathbb{R}$. Hence,
$\Big (\frac{1}{\sqrt{2}}a+\sqrt{2}b\Big )^2=\frac{1}{2}a^2+2b^2+2ab \ge 0$.
Combining that with the obvious identity for $(a+b)^2$, we get
$$\frac{1}{2}a^2+2b^2+[(a+b)^2-a^2-b^2] \ge 0$$
and equivalently,
$$(a+b)^2\ge a^2+b^2-\frac{1}{2}a^2-2b^2=\frac{1}{2}a^2-b^2.$$
Then
\begin{equation} \nonumber
|D_{s,y}u_{n}(t,x)|^2 \ge \frac{1}{2}a^2(x,y,s,t)-b^2(x,y,s,t).
\end{equation}
Next, integrating with respect to $s \in [t-\varepsilon,t]$ and $y \in [0,1]$, 
\begin{equation} \nonumber
\int_{t-\varepsilon}^{t}\int_{0}^{1} |D_{s,y}u_{n}(t,x)|^2 dyds \ge  \int_{t-\varepsilon}^{t}\int_{0}^{1} \Big (\frac{1}{2}a^2(x,y,s,t)-b^2(x,y,s,t) \Big ) dyds.
\end{equation}
Finally, 
\begin{equation} \label{autag54}
\begin{split}
\int_{0}^{t}\int_{0}^{1}|D_{s,y}u_{n}(t,x)|^2dyds &\ge \int_{t-\varepsilon}^{t}\int_{0}^{1}|D_{s,y}u_{n}(t,x)|^2 dyds \\
&\ge \frac{1}{2} \int_{t-\varepsilon}^{t}\int_{0}^{1} a^2(x,y,s,t) dyds - \int_{t-\varepsilon}^{t}\int_{0}^{1} b^2(x,y,s,t) dyds.
\end{split} 
\end{equation}
Now we need to find a lower and upper estimate for the terms in the above inequality. So, we try to maximize the term with the positive sign, i.e., $\int_{t-\varepsilon}^{t}\int_{0}^{1} a^2(x,y,s,t) dyds$, and respectively to minimize the term with negative sign, $\int_{t-\varepsilon}^{t}\int_{0}^{1} b^2(x,y,s,t) dyds$.
We have an  upper bound on the term $\int_{t-\varepsilon}^{t}\int_{0}^{1} b^2(x,y,s,t) dyds $ from Lemma $\ref{UD82}$. We showed that 
 \begin{equation} \label{b2boundcase22}
\begin{split}
&E\Big(\int_{t-\varepsilon}^{t}\int_{0}^{1} b^2(x,y,s,t)dyds\Big) \leq 2c \max\{  \varepsilon^{2}, \varepsilon^{\frac{1}{2}} \int_{t-\varepsilon}^{t} \sum_{k=1}^{\infty}  e^{-2k^2 \pi^2 (t-\theta)} \sin^{2}( k \pi x) d\theta \}. \\
\end{split} 
\end{equation}
For the lower estimate we have
\begin{equation} \nonumber
\begin{split}
 \int_{t-\varepsilon}^{t}\int_{0}^{1} a^2(x,y,s,t) dyds&= \int_{t-\varepsilon}^{t}\int_{0}^{1} |a(x,y,s,t)|^2 dyds \\
&= \int_{t-\varepsilon}^{t}\int_{0}^{1} |G_{t-s}(x,y)|^2|H_n(|u_{n}(s,y)|)\sigma(s,y,u_{n}(s,y))|^2 dyds.
\end{split} 
\end{equation}
Using assumption (3) on the diffusion coefficient $\sigma$ we have for any $x \in [0,1]$ and $\omega \in \Omega_{n}$
\begin{equation} \nonumber
\begin{split}
 \int_{t-\varepsilon}^{t}\int_{0}^{1} a^2(x,y,s,t) dyds \ge  c_0\int_{t-\varepsilon}^{t}\int_{0}^{1} |G_{t-s}(x,y)|^2 dyds,
\end{split} 
\end{equation}
since on $\Omega_{n}$ we have that $u_n=u$ and that $|u(s,y)|<n$ for all $(s,y)$ so that $H_{n}(|u_{n}(s,y)|)=1$.
Replacing the Green's function we have,
\begin{equation} \nonumber
\begin{split}
 \int_{t-\varepsilon}^{t}&\int_{0}^{1} a^2(x,y,s,t) dyds \ge c_{0} \int_{t-\varepsilon}^{t}\int_{0}^{1} \Big \{ 2\sum_{k=1}^{\infty}e^{-k^2 \pi^2 (t-s)} \sin{(k \pi x)}\sin{(k \pi y)} \Big \}^2 dyds.
\end{split} 
\end{equation}
Then the Parseval identity gives us that
\begin{equation} \label{a2bound}
\begin{split}
 \int_{t-\varepsilon}^{t}&\int_{0}^{1} a^2(x,y,s,t) dy ds \ge c_{1} \int_{t-\varepsilon}^{t}  \sum_{k=1}^{\infty} e^{-2k^2 \pi^2 (t-s)} \sin^2{(k \pi x)} ds. \\
\end{split} 
\end{equation}

\noindent We take $x \in (0,1)$ and then we have that $\sin^{2}(\pi x)>0$. Combining this with $(\ref{a2bound})$ we get 
\newline

$$
\begin{aligned}
&  \int_{t-\varepsilon}^{t}\int_{0}^{1} a^2(x,y,s,t) dy ds \geq 
c_{1}\sum_{k =1}^{\infty} \Big \{\sin^2{(k \pi x)}\int_{t-\varepsilon}^{t}  \exp^{-2k^2 \pi^2 (t-s)} ds \Big \} \geq  c_{1} \sin^2{(\pi x)}   \int_{t-\varepsilon}^{t}  \exp^{-2 \pi^2 (t-s)} ds  \\
&\geq  c_{1} \sin^{2}(\pi x) \frac{1}{2\pi^2}\Big (1-e^{-2 \pi^2 \varepsilon} \Big ) \geq \frac{c_{1}}{2} \sin^{2}(\pi x) \varepsilon:=c_{0} \varepsilon,
\end{aligned}
$$
where we used that  $1-e^{-x} \ge \frac{x}{2}$ as long as $x$ is small enough.
\newline 

Then, $(\ref{autag54})$ implies
\small
\begin{equation} \nonumber
\begin{split}
&P\bigg ( \omega \in \Omega_{n}: \int_{0}^{t}\int_{0}^{1}|D_{s,y}u_{n}(t,x)|^2dyds>0\bigg )\ge P\bigg ( \omega \in \Omega_{n}: \frac{1}{2} \int_{t-\varepsilon}^{t}\int_{0}^{1} a^2(x,y,s,t) dyds - \int_{t-\varepsilon}^{t}\int_{0}^{1} b^2(x,y,s,t) dyds>0 \bigg ) \\
&=P\bigg ( \omega \in \Omega_{n}: \int_{t-\varepsilon}^{t}\int_{0}^{1} b^2(x,y,s,t) dyds<\frac{1}{2} \int_{t-\varepsilon}^{t}\int_{0}^{1} a^2(x,y,s,t) dyds  \bigg ). \\
\end{split} 
\end{equation}

\noindent We will arrive to the desired result after separating cases.

\noindent Case 1: \\
Suppose that 

$$ \max\{  \varepsilon^{2}, \varepsilon^{\frac{1}{2}} \int_{t-\varepsilon}^{t} \sum_{k=1}^{\infty}  e^{-2k^2 \pi^2 (t-\theta)} \sin^{2}( k \pi x) d\theta \}= \varepsilon^{2}.$$
Then we have
\begin{equation} \label{b2boundcase1}
\begin{split}
&E\Big(\int_{t-\varepsilon}^{t}\int_{0}^{1} b^2(x,y,s,t)dyds\Big) \leq 2c \max\{  \varepsilon^{2}, \varepsilon^{\frac{1}{2}} \int_{t-\varepsilon}^{t} \sum_{k=1}^{\infty}  e^{-2k^2 \pi^2 (t-\theta)} \sin^{2}( k \pi x) d\theta \} \leq 2c \varepsilon^{2}
\end{split} 
\end{equation}
and moreover
\begin{equation} \nonumber
\begin{split}
P\bigg (\omega \in \Omega_{n}: \int_{0}^{t}\int_{0}^{1}|D_{s,y}u_{n}(t,x)|^2dyds>0\bigg ) &\ge P\bigg (\omega \in \Omega_{n}: \int_{t-\varepsilon}^{t}\int_{ 0}^{1} b^2(x,y,s,t) dyds< c_{0} \varepsilon  \bigg ). \\
\end{split} 
\end{equation}
Then we can apply Markov's inequality to get
\begin{equation} \nonumber
\begin{split}
P\bigg (\omega \in \Omega_{n}: \int_{0}^{t}\int_{0}^{1}|D_{s,y}u_{n}(t,x)|^2dyds>0\bigg ) &\ge P(\Omega_{n})- P\bigg ( \omega \in \Omega_{n}: \int_{t-\varepsilon}^{t}\int_{ 0}^{1} b^2(x,y,s,t) dyds \ge c_{0} \varepsilon  \bigg ) \\
&\ge P(\Omega_{n})-\mathbf{E}\Big \{\int_{t-\varepsilon}^{t}\int_{ 0}^{1} b^2(x,y,s,t) dyds  \Big \} \frac{1}{c_{0}\varepsilon }
\end{split} 
\end{equation}

$(\ref{b2boundcase1})$ gives that
\begin{equation} \nonumber
\begin{split}
P\bigg ( \omega \in \Omega_{n}: \int_{0}^{t}\int_{0}^{1}|D_{s,y}u_{n}(t,x)|^2dyds>0\bigg ) &\ge P(\Omega_{n})-2c\varepsilon^{2} \frac{1}{c_{0}\varepsilon }=P(\Omega_{n})-c_{3}\varepsilon \rightarrow P(\Omega_{n})
\end{split} 
\end{equation}
as $\varepsilon \rightarrow 0$. \\

Case 2: 
Otherwise if 
$$ \max\{  \varepsilon^{2}, \varepsilon^{\frac{1}{2}} \int_{t-\varepsilon}^{t} \sum_{k=1}^{\infty}  e^{-2k^2 \pi^2 (t-\theta)} \sin^{2}( k \pi x) d\theta \}= \varepsilon^{\frac{1}{2}} \int_{t-\varepsilon}^{t} \sum_{k=1}^{\infty}  e^{-2k^2 \pi^2 (t-\theta)} \sin^{2}( k \pi x) d\theta.$$ Then we have

\begin{equation} \label{b2boundcase2}
\begin{split}
&E\Big(\int_{t-\varepsilon}^{t}\int_{0}^{1} b^2(x,y,s,t)dyds\Big) \leq 2c \max\{  \varepsilon^{2}, \varepsilon^{\frac{1}{2}} \int_{t-\varepsilon}^{t} \sum_{k=1}^{\infty}  e^{-2k^2 \pi^2 (t-\theta)} \sin^{2}( k \pi x) d\theta \} \\
&\leq 2c\varepsilon^{\frac{1}{2}} \int_{t-\varepsilon}^{t} \sum_{k=1}^{\infty}  e^{-2k^2 \pi^2 (t-\theta)} \sin^{2}( k \pi x) d\theta.
\end{split} 
\end{equation}
According to $(\ref{a2bound})$
\small
\begin{equation} \nonumber
\begin{split}
&P\bigg (\omega \in \Omega_{n}: \int_{0}^{t}\int_{0}^{1}|D_{s,y}u_{n}(t,x)|^2dyds>0\bigg) \geq \\
&\ge P\bigg ( \omega \in \Omega_{n}: \int_{t-\varepsilon}^{t}\int_{ 0}^{1} b^2(x,y,s,t) dyds<  c_{1} \int_{t-\varepsilon}^{t}  \sum_{k=1}^{\infty} e^{-2k^2 \pi^2 (t-s)} \sin^2{(k \pi x)} ds \bigg ). \\
\end{split} 
\end{equation}

Then, we can apply Markov's inequality to get
\begin{equation} \nonumber
\begin{split}
&P\bigg ( \omega \in \Omega_{n}: \int_{0}^{t}\int_{0}^{1}|D_{s,y}u_{n}(t,x)|^2dyds>0\bigg ) \geq \\
&\ge P(\Omega_{n})- P\bigg ( \omega \in \Omega_{n}: \int_{t-\varepsilon}^{t}\int_{ 0}^{1} b^2(x,y,s,t) dyds \ge c_{1} \int_{t-\varepsilon}^{t}  \sum_{k=1}^{\infty} e^{-2k^2 \pi^2 (t-s)} \sin^2{(k \pi x)} ds \bigg ) \\
&\ge P(\Omega_{n})-\mathbf{E}\Big \{\int_{t-\varepsilon}^{t}\int_{ 0}^{1} b^2(x,y,s,t) dyds  \Big \} \frac{1}{c_{1} \int_{t-\varepsilon}^{t}  \sum_{k=1}^{\infty} e^{-2k^2 \pi^2 (t-s)} \sin^2{(k \pi x)} ds  }
\end{split} 
\end{equation}
and $(\ref{b2boundcase2})$ gives that
\small
\begin{equation} \nonumber
\begin{split}
&P\bigg( \omega \in \Omega_{n}: \int_{0}^{t}\int_{0}^{1}|D_{s,y}u_{n}(t,x)|^2dyds>0 \bigg)\geq \\
& \geq P(\Omega_{n})-\frac{2c\varepsilon^{\frac{1}{2}} \int_{t-\varepsilon}^{t} \sum_{k=1}^{\infty}  e^{-2k^2 \pi^2 (t-\theta)} \sin^{2}( k \pi x) d\theta}{ c_{1} \int_{t-\varepsilon}^{t}  \sum_{k=1}^{\infty} e^{-2k^2 \pi^2 (t-s)} \sin^2{(k \pi x)}ds  }=P(\Omega_{n})-c_{3}\varepsilon^{\frac{1}{2}} \rightarrow P(\Omega_{n})
\end{split} 
\end{equation}
as $\varepsilon \rightarrow 0$. 

\noindent Combining the two cases lets us conclude that \begin{equation} \nonumber
\begin{split}
&P\bigg( \omega \in \Omega_{n}: \int_0^t\int_0^1|D_{s,y}u_{n}(t,x)|^2dyds>0 \bigg)=P(\Omega_{n}).
\end{split} 
\end{equation}
\end{proof}

\begin{theorem} \label{dddddm9rjguhihgih}
Let $t \in (0,T), x \in (0,1)$ and $u$ be the mild solution to $(\ref{smr00})$ with coefficients satisfying the assumptions of Theorem $\ref{theorem21}$. Then the law of the random variable $u(t,x)$ is absolutely continuous with respect to the Lebesgue measure on $\mathbb{R}$. 
\end{theorem}

\begin{proof}
We will show the absolute continuity of the process $u(t,x)$ for $t \in (0,T)$ and $x \in (0,1)$. According to Theorem 2.1.3 in \cite{N} it suffices to show that  
\begin{equation} \nonumber
\|D_{\cdot,\cdot}u(t,x)\|^2_{L^2( [0,T] \times [0,1])}=\int_{0}^{t}\int_{0}^{1}|D_{s,y}u(t,x)|^2dyds>0,
\end{equation}
almost surely and that $u \in D_{1,1}^{loc}.$ Indeed we have that since $\Omega_n \uparrow \Omega$,
$$ \lim _{n \rightarrow \infty} P( \{ \omega \in \Omega_n :\|D_{s,y}u(t,x)\|_{L^{2}([0,T]\times[0,1])}>0 \})= P( \{ \omega \in \Omega :\|D_{s,y}u(t,x)\|_{L^{2}([0,T]\times [0,1])}>0 \}).$$
On the other hand, Theorem $(\ref{UD120})$ gives us $ P( \{ \omega \in \Omega_{n} :\|D_{s,y}u_{n}(t,x)\|_{L^{2}([0,T]\times [0,1])}>0 \})=P(\Omega_{n})$ and since $D_{s,y}u(t,x)=D_{s,y}u_{n}(t,x)$ on $\Omega_{n}$ we get 

\begin{equation} \nonumber
\begin{split}
&P( \{ \omega \in \Omega :\|D_{s,y}u(t,x)\|_{L^{2}([0,T]\times [0,1])}>0 \})= \lim _{n \rightarrow \infty} P( \{ \omega \in \Omega_n :\|D_{s,y}u(t,x)\|_{L^{2}([0,T]\times[0,1])}>0 \})= \\
&=\lim _{n \rightarrow \infty} P( \{ \omega \in \Omega_n :\|D_{s,y}u_{n}(t,x)\|_{L^{2}([0,T]\times[0,1])}>0 \})= \lim _{n \rightarrow \infty}P(\Omega_{n})=1.
\end{split}
\end{equation}

We conclude that \begin{equation} \nonumber
\|D_{\cdot,\cdot}u(t,x)\|^2_{L^2( [0,T] \times [0,1])}=\int_{0}^{t}\int_{0}^{1}|D_{s,y}u(t,x)|^2dyds>0,
\end{equation} almost surely for any $t \in (0,T), x \in (0,1)$.
Moreover, we have already proven that $u \in D_{1,2}^{loc}$, and  $D_{1,2}^{loc} \subseteq D_{1,1}^{loc}$ so the second condition is also satisfied. \\
We conclude that for any $t \in (0,T), x \in (0,1)$ the law of the random variable $u(t,x)$ is absolutely continuous with respect to the Lebesgue measure on $\mathbb{R}$. This completes the proof of Theorem $\ref{theorem21}$ from our main results section.
\end{proof}

 {\it Acknowledgement}. We would like to thank Konstantinos Tzirakis for helpful remarks and discussions. The research work is implemented in the framework of H.F.R.I call “Basic research Financing (Horizontal support of all Sciences)” under the National
Recovery and Resilience Plan “Greece 2.0” funded by the European Union – NextGenerationEU. (H.F.R.I. Project Number: 14910).
\bibliographystyle{plain}

\end{document}